\documentclass[10pt]{amsart}

\usepackage{amsmath,amsfonts,amssymb,dsfont}
\usepackage[usenames,dvipsnames]{color}
\usepackage{ulem}
\usepackage{graphicx, epsfig}
\usepackage{xspace}

\newcommand\R{\mathbb R}
\newcommand{\bE}{{\boldsymbol E}}

\newcommand{\bH}{{\boldsymbol H}}
\newcommand{\bu}{\boldsymbol u}
\newcommand{\bv}{\boldsymbol v}

\newcommand{\opmax}{\mathcal M}

\newcommand\D{\mathcal D}
\newcommand\dist{\operatorname{dist}}
\newcommand\spec{\sigma}
\renewcommand\div{\operatorname{div}}

\newcommand\x{\times}

\newcommand\llll{|\kern-1pt|\kern-1pt|}

\renewcommand{\L}{\mathcal{L}}
\newcommand{\E}{\mathcal{E}}

\newcommand{\dom}{\operatorname{D}}
\newcommand{\ud}{\mathrm{d}}
\renewcommand{\Re}{\operatorname{Re}}

\newcommand{\disc}{\mathrm{disc}}
\newcommand{\ess}{\mathrm{ess}}
\newcommand{\p}{\mathrm{p}}

\newcommand{\vecv}{\underline{v}}

\newcommand{\lstd}{\mathrm{F}}
\newcommand{\cbe}{\mathrm{c}}

\newcommand\1{{\ensuremath {\mathds 1} }}

\theoremstyle{remark}
\newtheorem{remark}{Remark}

\theoremstyle{plain}

\newtheorem{assumption}{Assumption}
\newtheorem{theorem}{Theorem}
\newtheorem{prop}[theorem]{Proposition}
\newtheorem{lemma}[theorem]{Lemma}
\newtheorem{cor}[theorem]{Corollary}

\DeclareMathOperator{\tr}{Tr}

\newcommand\curl{\operatorname{curl}}
\newcommand\nn{\mathbf{n}}

\newcommand\RR{\mathbb R}

\newcommand{\sobol}{\mathcal{H}}

\newcommand{\nea}{\mathfrak{n}}
\newcommand{\dnea}{\mathfrak{d}}

\newcommand{\bomega}{\partial \Omega}       

\newcommand{\zime}{Zimmermann-Mertins\xspace}
\newcommand{\zimandme}{Zimmermann and Mertins\xspace }

\begin{document}

\title[Eigenvalue enclosures]{Eigenvalue enclosures}
\author[G.R. Barrenechea]{Gabriel R. Barrenechea}
\address{Department of Mathematics and Statistics, University of Strathclyde, 26 
Richmond
Street, Glasgow G1 1XH, Scotland} 
\email{gabriel.barrenechea@strath.ac.uk}
\author[L. Boulton]{Lyonell Boulton}
\address{Department of Mathematics and Maxwell Institute for Mathematical
Sciences, Heriot-Watt University, Edinburgh, EH14 4AS, UK}
\email{L.Boulton@hw.ac.uk}
\author[N. Boussaid]{Nabile Boussa{\"\i}d}
\address{D\'epartement de Math\'ematiques, Universit\'e de Franche-Comt\'e,
Besan\c con, France}
\email{nboussai@univ-fcomte.fr}
\keywords{eigenvalue enclosures, spectral pollution,  finite element method,  
Maxwell equation}
\begin{abstract}
This paper is concerned with methods for numerical computation of eigenvalue 
enclosures. We examine in close detail the equivalence between an extension of 
the Lehmann-Maehly-Goerisch method  developed a few years ago by \zimandme, and 
a geometrically motivated method developed more recently by Davies and Plum. We 
extend various previously known results in the theory and establish explicit 
convergence estimates in both settings. The theoretical results are supported by 
two benchmark numerical experiments on the isotropic Maxwell eigenvalue problem.
\end{abstract}

\date{13th February 2014}
\maketitle
\tableofcontents


\section{Introduction}

Below we examine in close detail the equivalence between 
two pollution-free techniques for numerical computation of eigenvalue enclosures 
for general self-adjoint operators: a method considered a few years ago by 
\zimandme \cite{ZM95}, and a method developed more recently by Davies and Plum 
\cite{davies-plum}.  These  turn out to be highly robust and they can be applied 
to a wide variety of settings with minimal implementation difficulties.  

The approach of \zimandme is based on an extension of the 
Lehmann-Maehly-Goerisch method \cite{1985Goerisch,1974Weinberger} and it has 
proved to be highly successful in concrete numerical implementations. These 
include the computation of bounds for eigenvalues of the radially reduced 
magnetohydrodynamics operator \cite{ZM95,BouStr:2011man}, the study of 
complementary eigenvalue bounds for the Helmholtz equation 
\cite{Behnke:2001p2871} and the calculation of sloshing frequencies in the left 
definite case \cite{Behnke:2009p3097}.  

The method of Davies and Plum on the other hand, is based on a notion of 
approximated spectral distance which is highly geometrical in character. Its 
original formulation dates back to 
\cite{Davies:1998p3008,Davies:2000p2935,davies-plum}, and it is yet to be tested 
properly on models of dimension other than one.  
Our main motivation for the analysis conducted below, initiated with the results 
presented in \cite[Section~6]{davies-plum} where it is shown that both 
techniques are equivalent. Below we determine in a more precise manner the 
nature of this equivalence and examine their convergence properties. 

In Section~\ref{basic_method} we extend various canonical results from
\cite{davies-plum}. Notably, we include multiplicity counting 
(propositions~\ref{maybe_useless} and \ref{Prop:Generalization}) and
a description of how eigenfunctions are approximated 
(Proposition~\ref{eigenfunctions}). The method of \zimandme, on the other hand, 
is introduced in Section~\ref{zime_sec}. We derive the 
latter  in a self-contained manner independently from the work \cite{ZM95}. See 
Theorem~\ref{more_useful} and Corollary~\ref{corollary_4}. 

Section~\ref{ConverFn}  addresses the questions of convergence and upper bounds 
for residuals in both methods. 
The main statements in this respect are Theorem~\ref{thm:approx_square}, 
Corollary~\ref{cor_th11} and
Theorem~\ref{convergence}, where we formulate general convergence estimates with 
explicit bounds for a finite group of contiguous eigenvalues. 

Section~\ref{section_maxwell} is devoted to a concrete computational application 
in the spectral pollution regime.
For this purpose, we consider the model of the resonant cavity, for which it has 
been well-documented that nodal elements lead to spurious eigenvalues. 
Remarkably the present approach on nodal elements allows estimation of sharp 
eigenvalue bounds. A companion Comsol Multiphysics v4.3b Livelink code which was 
employed to produce some of the results presented in 
Section~\ref{section_maxwell}  as well as further numerical experiments on this 
model, is available in the appendix.


\section{Approximated local counting functions} \label{basic_method}
 
Let $A:\dom(A)\longrightarrow \mathcal{H}$  be a self-adjoint operator 
acting on a Hilbert space $\mathcal{H}$. Decompose the spectrum of $A$
in the usual fashion, as the union of discrete and  essential spectrum,
$\spec(A)=\spec_\disc(A)\cup \spec_\ess(A)$.  Let $J$ be any Borel subset of
$\R$. The spectral projector associated to $A$ is denoted by 
$\1_{J}(A)=\int_J \ud E_\lambda$. Hence $ \tr \1_{J}(A)=\dim 
\1_{J}(A)\mathcal{H}.$ We  write $\E_J(A)=\oplus_{\lambda\in
J} \ker(A-\lambda)$ with the convention $\E_\lambda(A)=\E_{\{\lambda\}}(A)$.
Generally $\E_J(A)\subseteq \1_{J}(A)\mathcal{H}$, however there is no reason 
for these 
two subspaces to be equal.

Let $t\in\R$.   Let $q_t:\dom(A)\x \dom(A)\longrightarrow \mathbb{C}$ be the 
closed bi-linear form
\begin{equation}\label{def_at}
 q_t(u,w)=\langle (A-t)u,(A-t)w\rangle \qquad \forall u,w\in\dom(A).
\end{equation}
For any $u\in \dom(A)$ we will constantly make use of the following 
$t$-dependant semi-norm, which is a norm if $t$ is not an eigenvalue,
\begin{equation}\label{def_bt}
|u|_t =q_t(u,u)^{1/2}=\|(A-t)u\|. 
\end{equation}
By virtue of the min-max principle, $q_t$ characterizes the spectrum which lies
near the origin of the positive operator $(A-t)^2$. In turn, this gives rise to 
a 
notion of local counting function at $t$ for the spectrum of $A$ as we will see 
next. 

Let 
\begin{align*}
    \dnea_j(t)& =   
    \inf_{ \substack{\dim V=j\\ V\subset \dom(A) }}
    \sup _{u\in V} \frac{|u|_t}{\|u\|} 
\end{align*}
so that $0\leq \dnea_j(t)\leq \dnea_{k}(t)$ for $j<k$. Then $\dnea_{1}(t)$ is 
the
Hausdorff distance from $t$ to $\spec(A)$,
 \begin{equation}
 \dnea_1(t)=\min\{ \lambda\in\spec(A): |\lambda-t|\} = \inf_{u\in 
\dom(A)}\frac{|u|_t}{\|u\|}.
 \end{equation}
Similarly $\dnea_j(t)$ are the distances from $t$ to the $j$th nearest point in
$\spec(A)$ counting multiplicity in a generalized sense. That is, stopping when
the essential spectrum is reached.
Moreover 
\[
      \dnea_{j}(t)=\dnea_{j-1}(t) \iff \left\{ \begin{aligned}  
       \text{either}\qquad   & \dim \E_{[t-\dnea_{j-1}(t),t+\dnea_{j-1}(t)]}(A) 
> j-1 \\
       \text{or} \qquad  &t+\dnea_{j-1}(t)\in \spec_\ess(A) \\
       \text{or}  \qquad &  t-\dnea_{j-1}(t)\in \spec_\ess(A).  \end{aligned}  
\right. 
\] 
 Without further mention, below we will always count spectral points of $A$ 
relative to $t$, 
regarding multiplicities in this generalized sense.

We now show how to extract certified information about $\spec(A)$ in the  
vicinity
of $t$ from the action of $A$ onto finite-dimensional trial subspaces $\L\subset
\dom(A)$, see \cite[Section~3]{Davies:1998p3008}.  For $j\leq n=\dim\L$, let
\begin{equation}   \label{defoffl}
    F_\L^j(t) =\min_{ \substack{\dim V=j\\ V\subset \L }} \max _{u\in V} 
\frac{|u|_t}{\|u\|}. 
\end{equation}
Then $0\leq F_\L^1(t)\leq \ldots \leq F_\L^n(t)$
and    $F_\L^j(t)\geq \dnea_j(t)$ for all $j=1,2,\ldots,n$. Since 
$[t-\dnea_j(t),t+\dnea_j(t)]\subseteq [t-F_\L^j(t),t+F_\L^j(t)]$,
there are at least $j$ spectral points  of $A$ in the segment
$\big[t-F^j_\L(t),t+F^j_\L(t)\big]$ including, possibly,  the essential 
spectrum. That is
\begin{equation} \label{multiplicity}
        \tr \1_{[t-F^j_\L(t),t+F^j_\L(t)]}(A) \geq j \qquad \forall
j=1,\ldots,n.
\end{equation}
Hence $F^j_\L(t)$ is an approximated local counting function for $\spec(A)$.

As a consequence of the triangle inequality, $F_\L^j$ is a Lipschitz continuous 
function such that
\begin{equation}
    |F_\L^j(t)-F_\L^j(s)|\leq |t-s| \qquad  \qquad \forall s,t\in \RR \quad 
\text{ and 
} \quad j=1,\ldots,n.\label{lipschitz_mult}
\end{equation}
Moreover, $F_\L^j(t)$ is the $j$th smallest eigenvalue $\mu$ of the 
non-negative 
weak problem: 
\begin{equation} \label{weak_form}
 \text{find }(\mu, u)\in [0,\infty)\times\L\! \setminus\!\{0\} \quad \text{
such that}\quad q_t(u,v) = \mu^2\langle u,v\rangle
\qquad \forall v\in \L. 
\end{equation}
Hence 
\begin{equation}   \label{defoffmaxmin}
    F_\L^j(t) =\max_{ \substack{\dim V=j-1\\ V\subset \L }} \min _{u\in 
\L\ominus
V} \frac{|u|_t}{\|u\|}=\max_{ \substack{\dim V=j-1\\ V\subset \mathcal{H}
}} \min _{u\in \L\ominus V} \frac{|u|_t}{\|u\|}. 
\end{equation}

\subsection{Optimal setting for detection of the spectrum}
As we show next, it is possible to detect the spectrum of $A$ to the left/right 
of $t$ by means 
of $F_\L^j$ in an optimal setting. 
This turns out to be a crucial ingredient in the formulation of the strategy 
proposed in
\cite{Davies:1998p3008,Davies:2000p2935,davies-plum}.  

The following notation simplifies various statements below. Let
\begin{align*}
   \nea^-_j(t)&=\sup \{s<t:\tr \1_{(s,t]}(A)\geq j\} \quad \text{and} \\
   \nea^+_j(t)&=\inf \{s>t:\tr \1_{[t,s)}(A)\geq j\}.
\end{align*} 
Then $\nea^\mp_j(t)$ is the $j$th point in $\spec(A)$ to the 
left$(-)$/right$(+)$ of $t$ counting multiplicities. 
Here $t\in \spec(A)$ is allowed and neither $t$ nor $\nea_1^{\mp}(t)$ have to 
be 
isolated from the rest of $\spec(A)$. 
Note that $\nea^-_j(t)=-\infty$ for $\tr \1_{(-\infty,t]}(A)<
j$ and $\nea^+_j(t)=+\infty$ for $\tr \1_{[t,+\infty)}(A)<
j$. Without further mention, all statements below regarding bounds on 
$\nea_j^\mp(t)$ will be void (hence redundant)  in either of these two cases.

\begin{prop}   \label{maybe_useless} 
Let $t^-<t<t^+$.  Then
\begin{equation} \label{enclosure_mult}
\begin{aligned}
F_\L^j(t^-)\leq t-t^- &\qquad \Rightarrow \qquad t^- - F^j_\L(t^-)\leq 
\nea_j^-(t) \\
F_\L^j(t^+)\leq t^+-t &\qquad \Rightarrow \qquad t^+ + F^j_\L(t^+)\geq 
\nea_j^+(t).
\end{aligned}
\end{equation}
Moreover, let $t_1^-<t_2^-<t<t_2^+<t_1^+$.  Then
\begin{equation} \label{inductive_step_mult}
\begin{aligned}
F^j_\L(t_{i}^-)\leq t-t_i^-\;\textrm{for}\;i=1,2 &\quad \Rightarrow \quad t_1^- 
-
F^j_\L(t_1^-)\leq t_2^- - F^j_\L(t_2^-)\leq \nea_j^-(t) \\
F^j_\L(t_i^+)\leq t_i^+-t \;\textrm{for}\;i=1,2&\quad \Rightarrow \quad t_1^+ + 
F^j_\L(t_1^+)\geq
t_2^+ + F^j_\L(t_2^+)\geq \nea_j^+(t).
\end{aligned}
\end{equation}
\end{prop}
\begin{proof}
We firstly show \eqref{enclosure_mult}. Suppose that $t\geq F_\L^j(t^-)+t^-$. 
Then 
\[
     \tr \1_{[t^- -F^j_\L (t^-),t]}(A)\geq j.
\]
Since $\nea_j^-(t)\leq \ldots \leq \nea_1^-(t)$ are the only spectral points
in the segment $[\nea_j^-(t),t]$, then necessarily
\[\nea^-_j(t)\in  [t^- -F^j_\L (t^-),t].\]  The bottom of 
\eqref{enclosure_mult} 
is shown in a similar fashion.

The second statement follows by observing that the maps $t\mapsto 
t\pm F_\L^j(t)$ are  monotonically increasing as a consequence of  
\eqref{lipschitz_mult}.
\end{proof}

The structure of the trial subspace $\L$ determines 
the existence of $t^\pm$ satisfying the hypothesis in \eqref{enclosure_mult}. 
If 
we expect to 
detect $\spec(A)$ at both sides of $t$, a necessary requirement on $\L$ should 
certainly be the condition 
\begin{equation} \label{mild_cond_L}
     \min_{u\in\L} \frac{\langle Au,u\rangle}{\langle u,u\rangle } <t<\max 
_{u\in\L} \frac{\langle Au,u\rangle}{\langle u,u\rangle }. 
\end{equation}
By virtue of lemmas~\ref{assumption_on_existence} and \ref{one_case_first} 
below, for $j=1$, the left hand side inequality of \eqref{mild_cond_L} 
implies the existence of $t^-$ and the right hand side inequality implies the 
existence of $t^+$, respectively. 

\begin{remark}   \label{op_tis}
From Proposition~\ref{maybe_useless} it follows that optimal lower bounds for 
$\nea_j^-(t)$  
are achieved by finding 
$\hat{t}^-_j\leq t$, the closer point to $t$, such that 
$F_\L^j(\hat{t}^-_j)=t-\hat{t}^-_j$.
Indeed, by virtue of \eqref{inductive_step_mult}, $ t^- - F_\L^j(t^-)\leq 
\hat{t}^-_j - 
F_\L^j(\hat{t}^-_j)\leq
\nea_j^-(t)$ for any other $t^-$ as in \eqref{enclosure_mult}. Similarly, 
optimal upper bounds 
for $\nea_j^+(t)$ are found by analogous means.  This observation will play a 
crucial role in Section~\ref{zime_sec}.
\end{remark}

The main result of this section is Proposition~\ref{maybe_useless}, which is 
central to the hierarchical method for
finding eigenvalue inclusions
examined a few years ago in \cite{Davies:1998p3008,Davies:2000p2935}. For
fixed $\L$ this method leads to bounds for eigenvalues which are far sharper
than those obtained from the obvious idea of estimating local
minima of $F^1_\L(t)$.  From an abstract perspective, 
Proposition~\ref{maybe_useless} provides an
intuitive insight on the mechanism for determining complementary bounds for
eigenvalues (in the left definite case, for example). The method proposed in 
\cite{Davies:1998p3008,Davies:2000p2935,davies-plum} is yet to
be explored more systematically in the practical setting, however in most 
circumstances the technique described in \cite{ZM95} is easier to implement.

\subsection{Geometrical properties of the first approximated counting function}
We now determine further geometrical properties of $F^1_\L$ and its connection 
to the spectral distance.  
Let the Hausdorff distances from $t\in \R$ to $\spec(A)\setminus(-\infty,t]$ 
and $\spec(A)\setminus[t,\infty)$, respectively, be given by
\begin{equation} \label{delta_plus_minus}
\begin{aligned}
\delta^+(t)&=\inf \{\mu-t:\mu\in \spec(A),\,\mu>t\} \qquad \text{and} \\
\delta^-(t)&=\inf \{t-\mu:\mu\in \spec(A),\,\mu<t\}.
\end{aligned}
\end{equation}
In general,  $t-\nea^-_1(t)\leq \delta^-(t)$ and $\nea^+_1(t)-t\leq 
\delta^+(t)$. In fact, 
 $|\nea^\pm_1(t)-t|= \delta^\pm(t)$ for $t\not\in \spec(A)$. However, these 
relations can be strict whenever $t\in \spec(A)$.
Indeed, $\nea_1^+(t)-t=\delta^{+}(t)$ iff there exists a decreasing sequence 
$t_n^+\in \spec(A)$ such that $t_n^+\downarrow t$, whereas 
$t-\nea_1^-(t)=\delta^{-}(t)$ iff there exists an increasing sequence 
$t_n^-\in \spec(A)$ such that $t_n^-\uparrow t$.

An emphasis in distinguishing $|\nea_1^\pm(t)-t|$ from
$\delta^\pm(t)$ seems unnecessary at this stage. However, this distinction in 
the
notation will be justified later on. Without further mention below we write 
$\delta^\pm(t)=\pm \infty$ to indicate 
that either of the sets on the right side of \eqref{delta_plus_minus} is empty.

Let $\lambda\in \spec(A)$ be an isolated point. If there exists a
non-vanishing  $u\in\L\cap \E_\lambda(A)$, then 
 \[
   \frac{|u|_s}{\|u\|}=|\lambda-s|=\dnea_1(s)   \qquad \forall s\in 
\left[\lambda-\frac{\delta^-(\lambda)}{2},\lambda+\frac{\delta^+(\lambda)}{2}
\right].
\]
According to the convergence analysis carried out in 
Section~\ref{subsec_ConverFn_2}, the smaller the angle between $\L$ and the 
spectral subspace $\E_\lambda(A)$, the closer the $F^1_\L(t)$ is to $\dnea_1(t)$ 
for 
$t\in\big(\lambda-\frac{\delta^-(\lambda)}{2},\lambda+\frac{\delta^+(\lambda)}{2
}\big)$.  
The special case of this angle being zero is described by the following lemma.

\begin{lemma}\label{lemma4}
For $\lambda\in \spec(A)$ isolated from the rest of the spectrum, the 
following statements are equivalent.
\begin{enumerate}
\item \label{no_hijo_0} There exists a minimizer $u\in \L$ of the right side of
\eqref{defoffl} for $j=1$, such that
$|u|_t=\dnea_1(t)$  for  a single $t\in
\big(\lambda-\frac{\delta^-(\lambda)}{2}, 
\lambda+\frac{\delta^+(\lambda)}{2}\big)$,
\item \label{no_hijo_1} $F^1_\L(t)=\dnea_1(t)$ for a single  $t\in
\big(\lambda-\frac{\delta^-(\lambda)}{2}, 
\lambda+\frac{\delta^+(\lambda)}{2}\big)$,
\item \label{no_hijo_2} $F^1_\L(s)=\dnea_1(s)$ for all $s\in
[\lambda-\frac{\delta^-(\lambda)}{2},\lambda+\frac{\delta^+(\lambda)}{2}]$,
\item \label{no_hijo_3} $\L\cap \E_\lambda(A)\not=\{0\}$.
\end{enumerate}
\end{lemma}
\begin{proof} 
 Since $\L$ is finite-dimensional, \ref{no_hijo_0} and
\ref{no_hijo_1} are equivalent by the definitions of $\dnea_1(t)$, $F^1_\L(t)$
and $q_t$. From the paragraph above the statement of the lemma it is clear that
\ref{no_hijo_3} $\Rightarrow$ \ref{no_hijo_2} $\Rightarrow$ \ref{no_hijo_1}.
Since $|u|_t/ \|u\|$ is the square root of the Rayleigh
quotient associated to the operator $(A-t)^2$, the fact that $\lambda$ is 
isolated combined with the
Rayleigh-Ritz principle, gives the implication
\ref{no_hijo_0}$\Rightarrow$\ref{no_hijo_3}. 
\end{proof}

As there can be a mixing of eigenspaces, it is not possible to 
replace \ref{no_hijo_1} in this lemma by an analogous statement including  
$t=\lambda\pm \frac{\delta^\pm(\lambda)}{2}$. If 
$\lambda'=\lambda+\delta^+(\lambda)$ is an eigenvalue, for example, then
$F^1_\L\left(\frac{\lambda+\lambda'}{2}\right)=\dnea_1\left(\frac{
\lambda+\lambda'}{2}\right)$ ensures that $\L$ contains elements of 
$\E_\lambda(A) \oplus\E_{\lambda'}(A)$. However it is not guaranteed to be 
orthogonal to either of these two  subspaces. 

\subsection{Geometrical properties of the subsequent approximated counting 
functions}

Various extensions of Lemma~\ref{lemma4} to  the case $j>1$ are possible, 
however it is difficult to write these results in a neat fashion. The 
proposition below is one such an extension.

The following generalization of Danskin's Theorem is a direct consequence of 
\cite[Theorem
D1]{BernhardRapaport}. Let $J\subset \R$ be an open segment. Denote by
\[
 \partial_{t}^{\pm} f(t)= \lim_{\tau \to 0^+}\pm\frac{f(t\pm \tau)-f(t)}{\tau},
\]
the one-side derivatives of a function $f:J\longrightarrow \R$. Let 
$\mathcal{V}$ be a compact topological space. For given $\mathcal{J}:J\times 
\mathcal{V}\longrightarrow \R$ we write
\[
      \tilde{\mathcal{J}}(t)=\max_{v\in \mathcal{V}} \mathcal{J}(t,v) \quad 
\text{ and } \quad
      \tilde{\mathcal{V}}(t)=\left\{\tilde{v}\in \mathcal{V}:
     \tilde{\mathcal{J}}(t)=\mathcal{J}(t,\tilde{v})\right\}.
\]

\begin{lemma} \label{Thm:Danskin} 
If the map $\mathcal{J}$ is upper semi-continuous and 
$\partial_{t}^{\pm}\mathcal{J}(t,v)$ exist for all $(t,v)\in J\times 
\mathcal{V}$, then also
$\partial_{t}^{\pm}\tilde{\mathcal{J}}(t)$ exist for all $t\in J$ and 
\begin{equation} \label{4}
     \partial_{t}^{\pm}\tilde{\mathcal{J}}(t)=\max_{\tilde{v}\in 
\tilde{\mathcal{V}}(t)}
     \partial_{t}^{\pm}\tilde{\mathcal{J}}(t,\tilde{v}).
\end{equation}
\end{lemma}

In the statement of this lemma, note that the left and right derivatives of both 
$\mathcal{J}$ and $\tilde{\mathcal{J}}$ might possibly be different.

\begin{prop}\label{Prop:Generalization}
Let $j=1,\ldots,n$ and $t\in\R$ be fixed. The following assertions are 
equivalent.
\begin{enumerate}
\item\label{n1} $|F_\L^j(t)-F_\L^j(s)|= |t-s|$ for some $s\not=t$.
\item\label{n2} There exists an open segment $J\subset \R$ containing $t$ in its 
closure, such that \[|F_\L^j(t)-F_\L^j(s)|= |t-s| \qquad \qquad \forall s\in 
\overline{J}.\] 
\item\label{n3} There exists an open segment $J\subset \R$ containing $t$ in its 
closure, such that
\[\forall s\in J,\text{ either} \quad\L\cap \E_{s+ F_\L^j(s)}\not=\{0\} \quad 
\text{or} \quad \L\cap \E_{s- F_\L^j(s)}(A)\not=\{0\}.\]
\end{enumerate}
\end{prop}
\begin{proof} \

\underline{\ref{n1} $\Rightarrow$ \ref{n2}}.
Assume \ref{n1}. Since $r\mapsto r\pm F^j_\L(r)$ are continuous and
monotonically increasing, then they have to be constant in the closure of
\[J=\{\tau t+(1-\tau)s:0<\tau< 1\}.\] This is precisely \ref{n2}.

\underline{\ref{n2} $\Rightarrow$ \ref{n3}}. 
Assume \ref{n2}. Then $s\mapsto F^j_\L(s)$ is differentiable in $J$ and its 
one-side derivatives are equal to $1$ or $-1$ in the whole of this interval. For
this part of the proof, we aim at applying \eqref{4}, in order to get another
expression for these derivatives.

Let $\mathcal{F}_j$ be the family of $(j\!-\!1)-$dimensional linear subspaces of 
$\L$.   Identify an orthonormal basis of $\L$ with the canonical basis of 
$\mathbb{C}^n$. Then any other orthonormal basis of $\L$ is represented by a 
matrix in $\mathrm{O}(n)$, the orthonormal group. By picking the first 
$(j\!-\!1)$ columns of these matrices, we cover all possible subspaces $V\in 
\mathcal{F}_j$. Indeed we just have to identify $(\vecv_1|\ldots|\vecv_{j-1})$ 
for $[\vecv_{kl}]_{kl=1}^n \in \mathrm{O}(n)$ with  
$V=\mathrm{Span}\{\vecv_k\}_{k=1}^{j-1}$. 

Let
\[
 \mathcal{K}_j=\Big\{(\vecv_1,\ldots,\vecv_{j-1}): [\vecv_{kl}]_{kl=1}^n \in 
\mathrm{O}(n) \Big\}\subset \underbrace{\mathbb{C}^n\times \ldots \times 
\mathbb{C}^n}_{j-1}.
\]
Then $\mathcal{K}_j$ is a compact subset in the product topology of the right 
hand side. According to \eqref{defoffmaxmin},
\[
    F^j_{\L}(s)=\max_{(\vecv_1,\ldots,\vecv_{j-1})\in \mathcal{K}_j} 
g(s;\vecv_1,\ldots,\vecv_{j-1})
\]
where
\[
 g(s;\vecv_1,\ldots,\vecv_{j-1})=\min_{\substack{(a_1,\ldots,a_{j-1})\in 
\mathbb{C}^{j-1}\\ \sum |a_k|^2=1}} \left|\sum a_k \tilde{v}_k \right|_s.
\]
Here we have used the correspondence between $\vecv_k\in\mathbb{C}^{n}$ and 
$\tilde{v}_k\in \L$ in the orthonormal basis set above. We write 
\[g(r,V)=g(r;\vecv_1,\ldots,\vecv_{j-1}) \quad \text{for} \quad 
V=\mathrm{Span}\{\tilde{\vecv}_k\}_{k=1}^{j-1} \in \mathcal{F}_j.
\]

The map $g:J\times\mathcal{K}_j\longrightarrow \R^+$ is the minimum of a 
differentiable function, so the hypotheses of Lemma~\ref{Thm:Danskin} are 
satisfied by $\mathcal{J}=-g$. Hence, by virtue of \eqref{4},
\[
     \partial_{s}^{\pm} g(s,V)=\min_{\substack{u\in
\L\ominus V, \|u\|=1\\ |u|_s=g(s,V)}}
\left(\frac{\Re l_s(u,u)}{|u|_s}\right)\,.
\]
As minima of continuous functions, $g(s,V)$ and $\partial_{s}^{\pm}g(s,V)$ are 
upper
semi-continuous. Therefore, a further application of Lemma~\ref{Thm:Danskin} 
yields 
\begin{align*}
  \partial_{s}^{\pm} F_\L^j(s)&=\max_{\substack{(\vecv_1,\ldots,\vecv_{j-1})\in
\mathcal{K}_j \\ g(s;\vecv_1,\ldots,\vecv_{j-1})=F_\L^j(s)}} \partial_s^\pm
g(s,\vecv_1,\ldots,\vecv_{j-1})\\
&=\max_{\substack{V\in \mathcal{F}_j \\
g(s,V)=F_\L^j(s)}} \min_{\substack{u\in
\L\ominus V, \|u\|=1\\ |u|_s=g(s,V)}}
\left(\frac{\Re l_s(u,u)}{|u|_s}\right).
\end{align*} 
Now, this shows that 
\[
\left|\max_{\substack{V\in \mathcal{F}_j \\
g(s,V)=F_\L^j(s)}} \min_{\substack{u\in
\L\ominus V, \|u\|=1\\ |u|_s=g(s,V)}}
\left(\frac{\Re 
l_s(u,u)}{|u|_s}\right)\right|=  1.
\]
As $\L$ is finite dimensional, there exists a vector
$u\in\L$ satisfying $|u|_s=F_\L^j(s)$ such that 
\[
 \frac{|\Re l_s(u,u)|}{|u|_s}=  1.
\]
Thus $|\Re \langle (A-s)u,u\rangle |=\langle (A-s)u,(A-s)u\rangle = F_\L^j(s)$.
Hence, according to the ``equality'' case in the Cauchy-Schwarz inequality, $u$ 
must be an eigenvector
of $A$ associated with either $s+F_\L^j(s)$ or $s-F_\L^j(s)$. This is precisely 
\ref{n3}.

\underline{\ref{n3} $\Rightarrow$ \ref{n1}}. 
Under the condition \ref{n3}, there exists an open segment $\tilde{J}\subseteq 
J$, possibly smaller, such that 
$t\in\overline{\tilde{J}}$ and $F^j_\L(s)=\dnea_j(s)$ for all $s\in \tilde{J}$.  
As $|\dnea_j(s)-\dnea_j(r)|=|s-r|$, then either \ref{n1} is immediate, or it 
follows by taking $r\to t$.
\end{proof}

As a consequence of this statement, we find the following extension of Proposition~\ref{maybe_useless} for  $t$ an eigenvalue.

\begin{cor} \label{prop_useless_in_spectrum}
Let $t\in \spec(A)$ be an eigenvalue of multiplicity $m$. Let $t^-<t<t^+$. If $\E_t(A)\cap \L = \{0\}$, then
\begin{equation} \label{enclosure_mult_degenerate}
\begin{aligned}
F_\L^j(t^-)\leq t-t^- &\qquad \Rightarrow \qquad t^- - F^j_\L(t^-)\leq 
\nea_{j+m}^-(t) \\
F_\L^j(t^+)\leq t^+-t &\qquad \Rightarrow \qquad t^+ + F^j_\L(t^+)\geq 
\nea_{j+m}^+(t).
\end{aligned}
\end{equation}
\end{cor}
\begin{proof}
According to \eqref{multiplicity}, 
\[
     \tr \1_{[t^- -F^j_\L (t^-),t^- +F^j_\L (t^-)]}(A)\geq j.
\]
Thus, if $t> F_\L^j(t^-)+t^-$, there is nothing to prove. 

Consider now the case $t=F_\L^j(t^-)+t^-$. If there exists $\tau<t^-$ such that
$t=F_\L^j(\tau)+\tau$, then from Proposition~\ref{Prop:Generalization}
there exists an open segment $J\subset \R$ containing $(\tau,t^-)$ such that
\[\forall s\in J,\text{ either} \quad\L\cap \E_{s+ F_\L^j(s)}\not=\{0\} \quad 
\text{or} \quad \L\cap \E_{s- F_\L^j(s)}(A)\not=\{0\}.\]
From the assumption, only the second alternative takes place, and necessarily
\[
\forall s\in (\tau,t^-),\, s - F_\L^j(s)\in \sigma_\p(A).
\]
Hence, as $s 
- F_\L^j(s)$ is continuous and $\mathcal{H}$ is separable, 
this function should be constant in the segment $(\tau,t^-)$. We also notice 
that due to monotonicity for any $s\in(\tau,t^-)$, 
$s+F_\L^j(s)=t^-$. Hence if $s\in (\tau,t^-)\mapsto s - F_\L^j(s)$ is constant, 
and equal to some value (say $v$), then $s$ is the 
midpoint between $t$ and $v$ for any $s\in (\tau,t^-)$, which is a 
contradiction with the fact that $\tau\neq t^-$. Hence
\[
t> F_\L^j(\tau)+\tau, \quad \forall \tau<t^-
\]
 and so 
\[
 \tau - F_\L^j(\tau)\leq \nea_{j+m}^-(t),
\]
for all $\tau<t^-$. Thus, by continuity, also
\[
 t^- - F_\L^j(t^-)\leq \nea_{j+m}^-(t).
\]

The bottom of 
\eqref{enclosure_mult_degenerate} 
is shown in a similar fashion.
\end{proof}


\subsection{Approximated eigenspaces}
We conclude this section by examining extensions of the implications
\ref{no_hijo_1} $\Rightarrow$ \ref{no_hijo_3} of Lemma \ref{lemma4} into a more
general context. In combination with the results of 
Section~\ref{zime_sec}, the next proposition shows how to obtain certified
information about spectral subspaces.

Here and below $\{u_j^t\}_{j=1}^n\subset \L$ will denote an orthonormal
family of eigenfunctions associated to the eigenvalues $\mu=F^j_\L(t)$ of the
weak problem \eqref{weak_form}.  In a suitable asymptotic regime for $\L$, the
angle between these eigenfunctions and the spectral subspaces  of $|A-t|$ in the
vicinity of the origin is controlled by a residual which is as small as
$\mathcal{O}\left(\sqrt{F_\L^j(t)-\dnea_j(t)}\right)$ for
$F_\L^j(t)-\dnea_j(t)\to 0$.  

\begin{assumption}   \label{asu1}
Unless otherwise specified, from now on we will always fix the parameter
$m\leq n=\dim\L$ and suppose that
\begin{equation}   \label{cond_J}
[t-\dnea_m(t),t+\dnea_m(t)]\cap \spec(A)\subseteq \spec_\disc(A).
\end{equation}
\end{assumption}
Set
\[   
\delta_j(t)=\dist\left[t,\spec(A)\setminus\left\{t\pm\dnea_k(t)\right\}_{k=1}
^j\right].
\]
By virtue of \eqref{cond_J}, $\delta_j(t)> \dnea_j(t)$ for  all $j\leq m$. 

\begin{remark}   \label{on_eves}
If $t=\frac{\nea_j^-(t)+\nea_j^+(t)}{2}$ for a given $j$, the vectors $\phi_j^t$ 
introduced in Proposition~\ref{eigenfunctions} and invoked subsequently, might 
not be eigenvectors of $A$ despite of the fact that $|A-t|\phi_j^t=\dnea_j(t) 
\phi_j^t$. 
However, in any other circumstance $\phi_j^t$ are eigenvectors of $A$.
\end{remark}

\begin{prop} \label{eigenfunctions}
Let $t\in \R$ and $j\in\{1,\ldots,m\}$. Assume that
$F_\L^j(t)-\dnea_j(t)$ is small enough so that $0<\varepsilon_j<1$ holds true
for the residuals constructed inductively as follows,
\begin{gather*}
\varepsilon_1=\sqrt{\frac{F_\L^1(t)^2-\dnea_1(t)^2}{\delta_1(t)^2-\dnea_1(t)^2}
}  \\   
\varepsilon_j=\sqrt{\frac{F^j_\L(t)^2-\dnea_j(t)^2}{\delta_j(t)^2-\dnea_j(t)^2}
    +\sum_{k=1}^{j-1}
    \frac{\varepsilon^2_k}{1-\varepsilon_k^2}   
\left(1+\frac{\dnea_j(t)^2-\dnea_k(t)^2}{\delta_j(t)^2-\dnea_j(t)^2}\right)}.
\end{gather*}  
Then, there exists an orthonormal basis $\{\phi_j^t\}_{j=1}^m$ of
$\E_{{[t-\dnea_m(t),t+\dnea_m(t)]}}(A)$
such that $\phi_j^t\in\E_{\{t- \dnea_j(t),t+ \dnea_j(t)\}}(A)$,
\begin{gather} \label{ubefu1}
\|u_j^t-\langle u_j^t,\phi_j^t\rangle \phi_j^{t}\|\leq \varepsilon_j \qquad
\text{and} \\ 
\label{ubefu2}
    \qquad |u_j^t-\langle u_j^t,\phi_j^t\rangle \phi_j^t|_t \leq
\sqrt{F^j_\L(t)^2-\dnea_j(t)^2+ \dnea_j(t)^2\varepsilon_j^2}.
\end{gather}
\end{prop}
\begin{proof}
 As it is clear from the context, in  this proof we suppress the index $t$ on
top of any vector. We write $\Pi_\mathcal{S}$ to denote the orthogonal
projection onto the subspace $\mathcal{S}$ with respect to the inner product
$\langle \cdot,\cdot \rangle$.

Let us first consider the case $j=1$. Let $\mathcal{S}_1=\E_{\{t-
\dnea_1(t),t+
\dnea_1(t)\}}\!(A)$, and decompose $u_1=\Pi_{\mathcal{S}_1}u_1+u_1^\perp$ where
$u_1^\perp \perp \mathcal{S}_1$. Since $A$ is self-adjoint,
\begin{equation} \label{casej1}
\begin{aligned}    
F^1_\L(t)^2&=\|(A-t)u_1\|^2=\dnea_1(t)^2\|\Pi_{\mathcal{S}_1}
u_1\|^2+\|(A-t)u_1^\perp\|^2.
\end{aligned}
\end{equation}
Hence 
\[
          F^1_\L(t)^2 \geq \dnea_1(t)^2(1-\|u_1^\perp\|^2)+\delta_1(t)^{{2}}
\|u_1^\perp\|^2.
\]
Since $\delta_1(t)> \dnea_1(t)$, clearing from this identity $\|u_1^\perp\|^2$
yields $\|u_1^\perp\| {\le} \varepsilon_1$.
Hence $\|\Pi_{\mathcal{S}_1}u_1\|^2 \geq 1-\varepsilon_1^2>0$. Let
\[
      \phi_1=\frac{1}{\|\Pi_{\mathcal{S}_1}u_1\|} {\Pi_{\mathcal{S}_1}u_1}
 \]
 so that $\|\Pi_{\mathcal{S}_1}u_1\|=|\langle u_1,\phi_1\rangle|$.
Then \eqref{ubefu1} holds immediately and  \eqref{ubefu2} is achieved by
clearing $\|(A-t)u_1^\perp\|^2$ from \eqref{casej1}.

We define the needed basis, and show  \eqref{ubefu1} and \eqref{ubefu2}, for $j$
up to $m$ inductively as follows. Set
\begin{gather*}
     \phi_j=\frac{1}{\|\Pi_{\mathcal{S}_j} u_j\|}\Pi_{\mathcal{S}_j} u_j
 \end{gather*}   
where $\mathcal{S}_j=\E_{\{t- \dnea_j(t),t+ \dnea_j(t)\}}\!(A)\ominus
\operatorname{Span}\{\phi_l\}_1^{j-1}$ and  $\Pi_{\mathcal{S}_j}u_j\not=0$, all
this for  $1\leq j \leq k-1$. 
Assume that \eqref{ubefu1}  and \eqref{ubefu2} hold true for $j$ up to $k-1$.
Define
$\mathcal{S}_k=\E_{\{t- \dnea_k(t),t+ \dnea_k(t)\}}\!(A)\ominus
\operatorname{Span}\{\phi_l\}_1^{k-1}$.
We first show that $\Pi_{\mathcal{S}_k}u_k\not=0$, and so we can define
\begin{equation} \label{defiphi}
     \phi_k=\frac{1}{\|\Pi_{\mathcal{S}_k} u_k\|}\Pi_{\mathcal{S}_k} u_k
 \end{equation}  
ensuring $\phi_k\perp \operatorname{Span}\{\phi_l\}_{l=1}^{k-1}$. After that we 
verify
 the validity of  \eqref{ubefu1} and \eqref{ubefu2} for $j=k$.

Decompose \[u_k=\Pi_{\mathcal{S}_k}u_k+\sum_{l=k-1}^{1} \langle
u_k,\phi_l\rangle \phi_l+u_k^\perp\]
where $u_k^\perp \perp \operatorname{Span}\{\phi_l\}_{l=1}^{k-1}\oplus
\mathcal{S}_k$. 
Then 
\begin{align*}
     F_\L^k(t)^2&= \dnea_k(t)^2\|\Pi_{\mathcal{S}_k}u_k \|^2+\sum_{l=k-1}^{1}
\dnea_l(t)^2|\langle u_k,\phi_l\rangle|^2 + \|(A-t)u_k^\perp\|^2 \\
     & \geq \dnea_k(t)^2 \|\Pi_{\mathcal{S}_k}u_k \|^2 +\sum_{l=k-1}^{1} 
\dnea_l(t)^2 |\langle u_k,\phi_l\rangle|^2 +\delta_k(t)^2\|u_k^\perp \|^2 \\
     &=\dnea_k(t)^2(1-\|u_k^\perp \|^2) + \sum_{l=k-1}^{1}
(\dnea_l(t)^2-\dnea_k(t)^2) |\langle u_k,\phi_l\rangle|^2 
+\delta_k(t)^2\|u_k^\perp \|^2.
\end{align*}
The conclusion  \eqref{ubefu1} up to $k-1$, implies $|\langle u_l,\phi_l\rangle
|^2\geq 1-\varepsilon_l^2$ for
$l=1,\ldots,k-1$. Since $\langle u_k,u_l\rangle=0$ for
$l\not=k$,
\[
    |\langle u_l,\phi_l\rangle | |\langle u_k,\phi_l\rangle |= 
      |\langle u_k,u_l-\langle u_l,\phi_l\rangle \phi_l\rangle |.
\] 
Then, the Cauchy-Schwarz inequality alongside with \eqref{ubefu1} yield
\begin{equation}
\label{casejgen}
    |\langle u_k,\phi_l\rangle |^2 \leq
\frac{\varepsilon_l^2}{1-\varepsilon_l^2}.
\end{equation}
Hence, since $\dnea_l(t)\le \dnea_k(t)$,
\[
     F_\L^k(t)^2 \geq \dnea_{{k}}(t)^2 + \sum_{l=k-1}^{1} 
(\dnea_l(t)^2-\dnea_k(t)^2) \frac{\varepsilon_l^2}{1-\varepsilon_l^2}
+(\delta_k(t)^2-\dnea_k(t)^2)\|u_k^\perp \|^2.
\]
Clearing $\|u_k^\perp \|^2$ from this inequality and combining with the validity
of \eqref{casejgen} and \eqref{ubefu1} up to $k-1$, yields 
$\Pi_{\mathcal{S}_k}u_k\not=0$. 

Let $\phi_k$ be as in \eqref{defiphi}. Then \eqref{ubefu1} is guaranteed for
$j=k$.
On the other hand, \eqref{ubefu1} up to $j=k$, \eqref{casejgen} and the identity
\[
    F^k_\L(t)^2 = \dnea_k(t)^2|\langle u_k,\phi_k\rangle |^2 + \|(A-t) 
(u_k-\langle u_k,\phi_k\rangle \phi_k)\|^2  ,
\]
yield \eqref{ubefu2} up to  $j=k$. \end{proof}


\section{Local bounds for eigenvalues}    \label{zime_sec}

Let  $t\in \R$ and $\L\subset \dom(A)$ be a specified trial subspace as above. 
Recall that $q_t$ is given by \eqref{def_at}.
Let $l_t:\dom(A)\x \dom(A)\longrightarrow \mathbb{C}$ be the (generally not
closed) bi-linear form associated to $(A-t)$,
\[
 l_t(u,w)=\langle (A-t)u,w\rangle \qquad \forall u,w\in\dom(A).
\]
Our next purpose is to characterize the optimal parameters $t^\pm$ in
Proposition~\ref{maybe_useless}  as described in Remark~\ref{op_tis} by means of
the following weak eigenvalue problem,
\begin{equation} \label{emrho}
\begin{aligned}
 &\textrm{find } u\in \L \setminus \{0\} \text{ and } \tau\in \RR  \text{ such
that}\\
 & \tau q_t(u,  v)  =
l_t( u,v)  \qquad \forall v\in \L  .
\end{aligned}
 \tag{Z$_t^{\L}$}
\end{equation}
This problem is central to the method of eigenvalue bounds calculation examined 
in \cite{ZM95}.

Let 
\[
\tau^-_{1}(t)\leq \ldots \leq \tau^-_{n^-}(t)<0 \qquad \text{ and }  \qquad
0<\tau^+_{n^+}(t)\leq \ldots \leq \tau^+_{1}(t),
\]
be the negative and positive eigenvalues of \eqref{emrho} respectively.
Here and below $n^\mp(t)$ is the number of these negative and
positive eigenvalues, which are both locally constant in $t$. 
Below we will denote eigenfunctions associated with $\tau^\mp_{j}(t)$ by
$u_j^\mp(t)$.   

\begin{assumption}  \label{plusminus}
For the purpose of clarity of exposition and without further mention, below we 
write most statements only for the case of ``lower bounds for the eigenvalues of 
$A$ which are to the left of $t$''. As the position of $t$ relative to the 
essential spectrum is irrelevant here, evidently this assumption does not 
restrict generality. The corresponding results regarding ``upper bounds for the 
eigenvalues of $A$ which are to the right of $t$''  can be  recovered by 
replacing $A$ by $-A$. 
\end{assumption}

The left side of the hypotheses \eqref{mild_cond_L} ensures the existence of
$\tau^-_1(t)$.  A more concrete connection with the framework of 
Section~\ref{basic_method} is made precise
in the following lemma. Its proof is straightforward, hence omitted. 

\begin{lemma}    \label{assumption_on_existence} 
The following conditions are equivalent,
\begin{itemize}
\item[{\it a}$^-$)] $F^1_\L(s)>t-s$ for all $s<t$
\item[{\it b}$^-$)] $\frac{\langle Au,u\rangle}{\langle u,u\rangle}>t$ for all
$u\in \L$
\item[{\it c}$^-$)] all the eigenvalues of \eqref{emrho} are positive.
\end{itemize}
\end{lemma}  

\begin{remark}  \label{remark_on_mult}
Let $\L=\operatorname{Span}\{b_j\}_{j=1}^n$. The matrix
$[q_t(b_j,b_k)]_{jk=1}^n$ is singular
if and only if $\E_t(A)\cap \L \neq \{0\}$. 
 On the other hand, the kernel of \eqref{emrho} might be non-empty. If $n_0(t)$
is the dimension
 of this kernel and $n_\infty(t)=\dim(\E_t(A)\cap \L)$, 
 then $n=n_\infty(t)+n_0(t)+n^-(t)+n^+(t)$. 
\end{remark}

\begin{assumption}  \label{asu2}
Note that $n_\infty(t)\geq 1$ if and only if $F_{\L}^j(t)=0$
for $j=1,\ldots,n_{\infty}(t)$. In this case the conclusions of
Lemma~\ref{one_case_first} and Theorem~\ref{more_useful} below become void.
In order to write our statements in a more transparent fashion, without further 
mention from now on we will suppose that 
\begin{equation} \label{easu2}
\L \cap \E_t(A)= \{0\}.
\end{equation}
\end{assumption}

By virtue of the next three results, finding the negative eigenvalues of 
\eqref{emrho}
is equivalent to  finding $s=\hat{t}^-_j \in\RR$ such that
\begin{equation}\label{GLj:fixedpoint} 
t-s= F^j_\L(s),
\end{equation}
and in this case  $\hat{t}^{-}_j=t+\frac{1}{2\tau_j^-(t)}$. 
It then follows from  Remark~\ref{op_tis} that \eqref{emrho}  encodes
information about the optimal bounds for the spectrum around $t$, achievable by
\eqref{inductive_step_mult} in Proposition~\ref{maybe_useless}.

\subsection{The eigenvalue immediately to the left}
We begin with the case $j=1$, see \cite[Theorem~11]{davies-plum}.

\begin{lemma}    \label{one_case_first}
Let $t\in \R$. The smallest eigenvalue  $\tau=\tau_1^-(t)$ of 
\eqref{emrho} is negative if and only if there exists $s<t$ such that
\eqref{GLj:fixedpoint} holds true. In this case $s=t+\frac{1}{2\tau^-_1(t)}$ and 
\[F^1_\L( s
)=-\frac{1}{2\tau^-_1(t)}= \frac{|u_1^-(t)|_s}{\|u_1^-(t)\|}\]
for $u=u^-_1(t)\in \L$ the corresponding eigenvector.
\end{lemma}

\begin{proof} 
For all $u\in \L$ and $s\in\R$,
\begin{align*}
q_s(u,u) - F_\L^1(s)^2\langle u,u \rangle = q_t(u,u) + 2(t-s)l_t(u,u)
+\left((t-s)^2-F^1_\L(s)^2\right) \langle u,u \rangle.
\end{align*}

Suppose that $F^1_\L(s)=t-s$.  Then
\begin{equation*}
q_s(u,u) - F^1_\L(s)^2\langle u,u \rangle = q_t(u,u) +2F^1_\L(s)l_t(u,u).
\end{equation*}
As the left side of this expression is non-negative, 
\[
    \frac{l_t(u,u)}{q_t(u,u)}\geq -\frac{1}{2F^1_\L(s)}
\]
for all $u\in \L\setminus \{0\}$ and the equality holds for some $u\in \L$.
Hence $-\frac{1}{2F^1_\L(s)}$
is the smallest eigenvalue of \eqref{emrho}, and thus necessarily equal to
$\tau_1^-(t)$.
In this case $s-F^1_\L(s)=t-2F^1_\L(s)=t+\frac{1}{\tau_1^-(t)}$. Here the vector
$u$ for which equality is achieved
is exactly $u=u^-_1(t)$.

Conversely, let $\tau^-_1(t)$ and $u^-_1(t)$ be as stated. Then \[
    \tau^-_1(t)\leq  \frac{l_t(u,u)}{q_t(u,u)}
\]
for all $u\in \L$ with equality for $u=u^-_1(t)$.  Re-arranging this expression
yields
\[
   q_t(u,u) - \frac{1}{\tau^-_1(t)} l_t(u,u) \geq 0
\]
for all $u\in \L$ with equality for $u=u^-_1(t)$. The substitution
$t=s-\frac{1}{2\tau^-_1(t)}$ then yields
\[
   q_t(u,u) -\frac{1}{(2\tau_1^-(t))^2} \langle u,u\rangle \geq 0
\]
for all $u\in \L$. The equality holds for $u=u^-_1(t)$. This expression further
re-arranges as
\[
    \frac{|u|_s^2}{\|u\|^2}\geq \frac{1}{(2\tau^-_1(t))^2}.
\]
Hence $F^1_\L(s)^2=\frac{1}{(2\tau_1^-(t))^2}$, as needed.
\end{proof}

\subsection{Further eigenvalues}
An extension to $j\not=1$ is now found by induction.

\begin{theorem}  \label{more_useful} 
Let $t\in \R$ and $1\leq j\leq n$ be fixed.  The number of negative eigenvalues 
$n^-(t)$ in  \eqref{emrho} is
greater than or equal to $j$  if and only if
\[
      \frac{\langle Au,u\rangle}{\langle u,u \rangle}<  t \qquad \text{for some}
\quad u\in
      \L \ominus \operatorname{Span}\{ u^-_1(t),\ldots,u^-_{j-1}(t) \}.
\]
Assuming this holds true, then $\tau=\tau^-_j(t)$  and $u=u_j^-(t)$ are
solutions of \eqref{emrho} if and only if
\[
     F^j_\L\left( t+\frac{1}{2\tau_j^-(t)} \right)=- \frac{1}{2\tau_j^-(t)} = 
     \frac{\left| u^-_j(t)\right|_{t+\frac{1}{2\tau_j^-(t)}}}{\|u^-_j(t)\|}.
\]
\end{theorem}
\begin{proof}
For $j=1$ the statements are Lemma~\ref{one_case_first} taking into 
consideration \eqref{mild_cond_L}. For $j>1$,
 due to the self-adjointness of the eigenproblem \eqref{emrho},
it is enough to
apply again Lemma~\ref{one_case_first} by fixing $\tilde{\L}=\L \ominus 
\operatorname{Span}\{ u^-_1(t),\ldots,u^-_{j-1}(t) \}$ as trial spaces. 
Note that the negative eigenvalues of $(\mathrm{Z}^{\tilde{\L}}_t)$ are those of
\eqref{emrho} except for $\tau_1^-(t),\ldots,\tau_{j-1}^-(t)$.
\end{proof}

A neat procedure for finding certified spectral bounds for $A$, as
described in \cite{ZM95}, can now be deduced from Theorem~\ref{more_useful}. 
By virtue of Proposition~\ref{maybe_useless} and Remark~\ref{op_tis}, this
procedure turns out to be optimal in the context of the approximated counting
functions discussed in Section~\ref{basic_method}, see 
\cite[Section~6]{davies-plum}. 
We summarize the core statement as follows.

\begin{cor} \label{corollary_4}
For all $t\in \RR$ and $j\in \{1,\ldots,
n^\pm(t)\}$,
\begin{equation}   \label{encl_gen}
    t+\frac{1}{\tau_j^-(t)}\leq \nea_j^-(t) \qquad \text{and} \qquad
    \nea_j^+(t) \leq t+\frac{1}{\tau_j^+(t)}.
\end{equation}
\end{cor}

In recent years, numerical techniques based on this statement have been designed 
to successfully compute eigenvalues for the radially reduced 
magnetohydrodynamics operator \cite{ZM95,BouStr:2011man}, the Helmholtz equation 
\cite{Behnke:2001p2871} and the calculation of sloshing frequencies in the left 
definite case \cite{Behnke:2009p3097}. We will explore the case of the Maxwell 
operator in sections~\ref{section_maxwell}.


\section{Convergence and error estimates}
\label{ConverFn}\label{subsec_ConverFn_2}

Our first goal in this section will be to show that,
if $\L$ captures an eigenspace of $A$
within a certain order of precision $\mathcal{O}(\varepsilon)$ as
specified below, then the bounds which follow from
Proposition~\ref{maybe_useless}  are 
\begin{enumerate}
\item \label{convergence_a} at least within $\mathcal{O}(\varepsilon)$ from the
true spectral data for any $t\in \R$,
\item \label{convergence_b} within $\mathcal{O}(\varepsilon^2)$ for $t\not\in
\spec(A)$.
\end{enumerate}
This will be the content of theorems~\ref{linear_order} and 
\ref{thm:approx_square}, and Corollary~\ref{cor_th11}.
We will then show that, in turns, the estimates \eqref{encl_gen} have always 
residual of size
$\mathcal{O}(\varepsilon^2)$ for any $t\in \R$. See  Theorem~\ref{convergence}.
In the spectral approximation literature this property is known as optimal order
of convergence/exactness, see \cite[Chapter~6]{1983Chatelin} or 
\cite{1974Weinberger}.

Recall Remark~\ref{on_eves}, and the assumptions~\ref{asu1} and~\ref{asu2}. 
Below  $\{\phi_j^t\}_{j=1}^m$ 
denotes an orthonormal set of eigenvectors of
$\E_{[t-\dnea_m(t),t+\dnea_m(t)]}(A)$ which is ordered so that
\[
      |A-t|\phi_j^t=\dnea_j(t) \phi_j^t \qquad \text{for} \qquad j=1,\ldots, m.
\]
Whenever $0<\varepsilon_j<1$ is small, as specified below, the trial subspace
$\L\subset \dom(A)$ will be assumed to be close to
$\operatorname{Span}\{\phi_j^t\}_{j=1}^m$ in the sense that there exist
$w_j^t\in \L $  such that
\begin{align} \label{cond_approx_0} \tag{A$_0$}
\|w_j^t-\phi_j^t\|&\leq\varepsilon_j  \qquad \text{and} \\
\label{cond_approx_1} \tag{A$_1$}
|w_j^t-\phi_j^t|_t&\leq \varepsilon_j .
\end{align} 
We have split this condition into two, in order to highlight the fact
that some times only \eqref{cond_approx_1} is required. Unless otherwise
specified, the index $j$ runs from $1$~to~$m$.

From \eqref{cond_J} it follows that the family 
$\{\phi_j^s\}_{j=1}^m\subset\E_{[t-\dnea_m(t),t+\dnea_m(t)]}(A)$ 
and the family  $\{w_j^s\}_{j=1}^m\subset \L$ above can always be chosen 
piecewise constant for
$s$ in a neighbourhood of $t$. Moreover, they can be chosen so that jumps only
occur at $s\in\spec(A)$. 

\begin{assumption} 
\label{asu3} Without further mention all $t$-dependant vectors below will be 
assumed to be locally constant in $t$
with jumps only at the spectrum of $A$.
\end{assumption}

A set $\{w_j^t\}_{j=1}^m$ subject to
\eqref{cond_approx_0}-\eqref{cond_approx_1} is not generally orthonormal.
However, according to the next lemma, it can always be substituted by an 
orthonormal set,  provided $\varepsilon_j$ is
small enough. 

\begin{lemma} \label{onb_approx}
There exists a constant $C>0$ independent of $\L$ ensuring the following. If
$\{w_j^t\}_{j=1}^m\subset \L $ is such that
\eqref{cond_approx_0}-\eqref{cond_approx_1} hold for all $\varepsilon_j$ such 
that \[\varepsilon=\sqrt{\sum_{j=1}^m\varepsilon_j^2}<\frac1{\sqrt{m}},\] then 
there
is a set $\{v_j^t\}_{j=1}^m\subset \L$  orthonormal in the inner product
$\langle \cdot,\cdot\rangle$ such that
\[
       |v_j^t-\phi_j^t|_t + \|v_j^t-\phi_j^t\|< C \varepsilon.
\]  
\end{lemma}

\begin{proof}
As it is clear from the context, in  this proof we suppress the index $t$ on
top of any vector. The desired conclusion is achieved by applying the 
Gram-Schmidt procedure.  Let 
$G=[\langle w_k,w_l\rangle]_{kl=1}^m \in \mathbb{C}^{m\times m}$ 
be the Gram matrix associated to $\{w_j\}$.
Set
$$v_j =\sum_{k=1}^m(G^{-1/2})_{kj}\;w_k.$$
Then
\begin{align*}
 \|G-I\|&\leq\sqrt{\sum_{k,l=1}^m |\langle w_k,w_l\rangle-\langle
\phi_k,\phi_l\rangle|^2}\\
&\leq\sqrt{2\sum_{k,l=1}^m
\|w_k-\phi_k\|^2(\|w_l\|+\|\phi_l\|)^2}\\
&\leq
\sqrt{2}(2+\varepsilon)\varepsilon.
\end{align*}

Since
\begin{align*}
    \|v_j-w_j\|^2 & = \left\|\sum_{k=1}^m(G^{-1/2}-I)_{kj}\;w_k\right\|^2\\
    &= \sum_{k,l=1}^m(G^{-1/2}-I)_{kj}
\overline{(G^{-1/2}-I)_{lj}}\langle w_k, w_{l} \rangle\\
&= \sum_{k=1}^m (G^{-1/2}-I)_{kj} \overline{\left( \sum_{l=1}^m 
G_{kl}(G^{-1/2}-I)_{lj}\right)}\\
&= \sum_{k=1}^m (G^{-1/2}-I)_{kj}(G^{1/2}-G)_{jk}\\
&= \left( (I-G^{1/2})^2\right)_{jj}
\end{align*}
then
\begin{equation*}\label{GS-1}
 \|v_j-w_j\|\le \|I-G^{1/2}\|.
\end{equation*}
As $G^{1/2}$ is a positive-definite matrix, for every $\vecv \in\mathbb{C}^m$ we 
have
\[ \|(G^{1/2}+I)\vecv\|^2=\|G^{1/2}\vecv\|^2+2\langle G^{1/2}\vecv,\vecv \rangle 
+\|\vecv\|^2 \ge \|\vecv\|^2. \]
Then  $\det(I+G^{1/2})\not=0$ and $\|(I+G^{1/2})^{-1}\| \le 1$. Hence
\begin{equation} \label{GS-2}
\|v_j-w_j\| \leq \|(I-G)(I+G^{1/2})^{-1}\| \le \|I-G\|\,\|(I+G^{1/2})^{-1}\|\le 
(2+\varepsilon)\varepsilon\,.
\end{equation}

Now, identify $\vecv=(v_1,\ldots,v_m)\in{\mathbb C}^m$ with $v=\sum_{k=1}^m 
v_k\phi_k$. As
\[
  \|G^{1/2} \vecv\| =\left\|\sum_{j=1}^m \langle v, \phi_j\rangle w_j 
\right\|\geq \|v\| - \left\|\sum_{j=1}^m \langle v, \phi_j\rangle 
(w_j-\phi_j)\right\| \geq (1-\varepsilon)\|\vecv\|
 \]
then
 \[
  \|G^{-1/2}\| \leq \frac{1}{1-\varepsilon}.
\]
Hence 
\begin{align}
|v_j-w_j|_t &\leq \sum_{k=1}^m|(G^{-1/2}-I)_{jk}||w_k|_t  \nonumber\\
&\leq \sum_{k=1}^m|(G^{-1/2}-I)_{jk}|(\varepsilon_k+\dnea_k(t)) \nonumber\\
&\leq
\sum_{k,l=1}^m|(G^{-1/2})_{kl}||(G^{1/2}-I)_{lj}|(\varepsilon_k+\dnea_k(t)) 
\nonumber\\
&\leq\frac{\sqrt{m}(\varepsilon+\dnea_m(t))
(2+\varepsilon)}{1-\varepsilon}\varepsilon. \label{GS-3}
\end{align}
The desired conclusion follows from \eqref{GS-2} and \eqref{GS-3}.
\end{proof}

\subsection{Convergence of the approximated local counting function}
The next theorem addresses the claim \ref{convergence_a} made at the beginning
of this section. According to Lemma~\ref{onb_approx},  
in order to examine the asymptotic behaviour of $F^j_\L(t)$ as $\varepsilon_j\to
0$ under the constraints
\eqref{cond_approx_0}-\eqref{cond_approx_1}, we can assume without loss of
generality that the trial vectors $w_j^t$ form an orthonormal set
in the inner product $\langle\cdot,\cdot\rangle$.

\begin{theorem} \label{linear_order}
Let $\{w_j^t\}_{j=1}^m\subset \L$ be a family of vectors
which is orthonormal in the inner product $\langle\cdot,\cdot \rangle$ and
satisfies \eqref{cond_approx_1}.
Then
\[
       F_{\L}^j(t)-\dnea_j(t) \leq \left( \sum_{k=1}^j \varepsilon_k^2 \right)
^{1/2}\qquad \quad \forall j=1,\ldots,m.
\] 
\end{theorem}
\begin{proof}  From the min-max principle we obtain
\begin{align*}
      F_{\L}^j(t)&\leq \max_{\sum |c_k|^2=1} \left|\sum_{k=1}^j c_k w_k
\right|_t \\
      &\leq \max_{\sum |c_k|^2=1} \left|\sum_{k=1}^j c_k (w_k-\phi_k) \right|_t+
      \max_{\sum |c_k|^2=1} \left|\sum_{k=1}^j c_k \phi_k\right|_t \\
      &= \max_{\sum |c_k|^2=1} \left|\sum_{k=1}^j c_k
(w_k-\phi_k)\right|_t+\dnea_j(t).
\end{align*}
This gives
\begin{align*}
    F_{\L}^j(t)-\dnea_j(t)&\leq 
    \max_{\sum |c_k|^2=1} \sum_{k=1}^j |c_k| |w_k-\phi_k|_t \\
    &\leq   \max_{\sum |c_k|^2=1} \left(\sum_{k=1}^j |c_k|^2\right)^{1/2}
    \left( \sum_{k=1}^j |w_k-\phi_k|_t^2 \right)^{1/2} \leq \left( \sum_{k=1}^j
\varepsilon_k^2 \right) ^{1/2}
\end{align*}
as needed.
\end{proof}

In terms of order of approximation, Theorem~\ref{linear_order} will be 
superseded by Theorem~\ref{thm:approx_square} for $t\not\in \spec(A)$. However, 
if $t\in \spec(A)$, 
the trial space $\L$ can be chosen so that $F_{\L}^1(t)-\dnea_1(t)$ is linear in 
$\varepsilon_1$. Indeed, fixing any non-zero $u\in\dom(A)$ and 
$\L=\operatorname{Span}\{u\}$, yields
$F^1_{\L}(t)-\dnea_1(t)=F^1_{\L}(t)=\varepsilon_1$.
This shows that Theorem~\ref{linear_order} is optimal, upon the presumption that 
$t$  is arbitrary.

The next theorem addresses the claim \ref{convergence_b} made at the beginning
of this section.  Its proof is reminiscent of that of
\cite[Theorem~6.1]{1973Strang}.

\begin{theorem} \label{thm:approx_square}
Let $t\not\in\spec(A)$. Suppose that the $\varepsilon_j$ in
\eqref{cond_approx_1} are such that
\begin{equation}\label{epsilonk-chico}
\sum_{j=1}^m \varepsilon_j^2<\frac{\dnea_1(t)^2}{6}.
\end{equation}
Then,  
\begin{equation} \label{quadratic_conv}
    F^j_\mathcal{L}(t)-\dnea_j(t)\leq 3\frac{\dnea_j(t)}{\dnea_1(t)^2 }
\sum_{k=1}^j \varepsilon_k^2
    \qquad  \forall j=1,\ldots,m. 
\end{equation}
\end{theorem}
\begin{proof}
Since $t\not\in \spec(A)$, then $(\dom(A),q_t(\cdot,\cdot))$ is a Hilbert space.
Let $P_\L:\dom(A)\longrightarrow \L$ be the orthogonal projection onto $\L$ with
respect to the inner product $q_t(\cdot,\cdot)$, so that
\[
      q_t(u-P_{\L}u,v)=0 \qquad \forall v\in \L.
\]
Then $|u|_t^2=|P_\L u|_t^2+|u-P_\L u|_t^2$ for all $u\in \dom(A)$ and
$|u-P_{\L}u|_t \leq |u-v|_t$ for all $v\in\L$. Hence
\begin{equation} \label{cond_approx_2}
|\phi_k-P_{\L} \phi_k|_t \leq \varepsilon_k \qquad \qquad \forall k=1,\ldots,m.
\end{equation}

Let $\E_j=\mathrm{Span}\big\{\phi_k\}_{k=1}^j$.
Define
\begin{align*}
     \mathcal{F}_j&=\{\phi \in \E_j  : \|\phi\|=1 \big\} \qquad \text{and} \\
     \mu_{\L}^j(t)&=\max_{\phi \in \mathcal{F}_j} \left | 2 \Re \langle \phi,
\phi-P_{\L} \phi\rangle -\|\phi-P_{\L}\phi \|^2
     \right|.
\end{align*}
Here $\mu_{\L}^j$ depends on $t$, as $P_\L$ does. We first show that, under
hypothesis \eqref{epsilonk-chico}, $\mu_\L^j(t)<\frac{1}{2}$. Indeed,
 given $\phi\in \mathcal{F}_j$ we decompose it as $\phi=\sum_{k=1}^j c_k
\phi_k$. Then
\begin{align}
    |\langle \phi,\phi-P_\L \phi \rangle | &=
    \left|\sum_{k=1}^j c_k \langle \phi_k, \phi-P_\L \phi \rangle \right| 
    =\left|\sum_{k=1}^j \frac{c_k}{\dnea_k(t)^2} q_t(\phi_k, \phi-P_\L \phi) 
    \right|  \nonumber\\
    &=\left|q_t\left( \sum_{k=1}^j \frac{c_k}{\dnea_k(t)^2}\phi_k, \phi-P_\L
\phi\right) \right|  \nonumber \\
    &= \left|q_t\left( \sum_{k=1}^j \frac{c_k}{\dnea_k(t)^2}(\phi_k-P_\L\phi_k),
\phi-P_\L \phi\right)\right|  \nonumber \\
    &\leq \left|  \sum_{k=1}^j \frac{c_k}{\dnea_k(t)^2}(\phi_k-P_\L\phi_k)
\right|_t\, \left|  \sum_{k=1}^j c_k(\phi_k-P_\L\phi_k) \right|_t.
\label{auxi-0}
\end{align} 
For each multiplying term in the latter expression, the triangle and 
Cauchy-Schwarz's 
inequalities yield (take $\alpha_k=c_k$ or $\alpha_k=\frac{c_k}{\dnea_k(t)^2}$)
\begin{align}
       \left| \sum_{k=1}^j \alpha_k (\phi_k-P_\L \phi_k) \right|_t
 &\leq \sum_{k=1}^j |\alpha_k|\,  |\phi_k-P_\L \phi_k|_t \nonumber\\
 &\leq \left(\sum_{k=1}^j |\alpha_k|^2\right)^{1/2} 
       \left( \sum_{k=1}^j |\phi_k-P_\L \phi_k|_t^2 \right)^{1/2}.\label{auxi-1}
\end{align}
Then
\begin{equation}
\begin{aligned}
 \left| 2\Re  \langle\phi,\phi-P_\L \phi \rangle\right|  &\leq
2\left(\sum_{k=1}^j \frac{|c_k|^2}{\dnea_k(t)^4}\right)^{1/2}
 \left(\sum_{k=1}^j |c_k|^2\right)^{1/2} \sum_{k=1}^j \varepsilon_k^2  \\
 & \leq \frac{2}{\dnea_1(t)^2}\sum_{k=1}^j \varepsilon_k^2\label{cota-1}
\end{aligned}
\end{equation}
for all $\phi\in \mathcal{F}_j$. 

The other term in the expression for $\mu_\L^j(t)$ has an upper bound found as
follows. According to the min-max principle
\begin{equation}
\|\phi-P_\L\phi\|^2 \le \frac{1}{\dnea_1(t)^2}q_t(\phi-P_\L\phi,\phi-P_\L\phi).
\end{equation}
Therefore, by repeating analogous steps as in \eqref{auxi-0} and \eqref{auxi-1}, 
we
get
\begin{align}
\|\phi-P_\L\phi\|^2 &\le \frac{1}{\dnea_1(t)^2}\sum_{k=1}^jc_k
q_t(\phi_k-P_\L\phi_k,\phi-P_\L\phi) \nonumber\\
&= q_t\left( \sum_{k=1}^j
\frac{c_k}{\dnea_1(t)^2}(\phi_k-P_\L\phi_k),\phi-P_\L\phi\right)\nonumber\\
&= q_t\left( \sum_{k=1}^j
\frac{c_k}{\dnea_1(t)^2}(\phi_k-P_\L\phi_k),\sum_{l=1}
^jc_l(\phi_l-P_\L\phi_l)\right)\nonumber\\
&\le \frac{1}{\dnea_1(t)^2}\sum_{k=1}^j \varepsilon_k^2\,.\label{cota-2}
\end{align}
Hence, from \eqref{cota-1} and \eqref{cota-2}, 
\begin{equation}\label{bound_on_mu_L}
\mu_\L^j(t)\le \frac{3}{\dnea_1(t)^2}\sum_{k=1}^j\varepsilon_k^2 < \frac{1}{2}
\end{equation}
as a consequence of \eqref{epsilonk-chico}. 

Next, observe that $\dim (P_\L  \E_j) = j$. Indeed $P_\L \psi=0$
for $\|\psi\|=1$ would imply 
\[
      \mu_\L ^j(t)\geq \left | 2 \Re \langle \psi, \psi-P_{\L} \psi\rangle
-\|\psi-P_{\L}\psi \|^2
     \right| =\|\psi\|^2=1,
\] 
which would contradict the fact that $\mu_\L ^j(t)<1$. Then,
\[
      F^j_\L (t)^2\leq \max _{u\in P_\L \E_j}\frac{|u|_t^2}{\|u\|^2} 
      =\max_{\phi\in \E_j}\frac{|P_\L \phi|_t^2}{\|P_\L \phi\|^2}
      =\max_{\phi\in \mathcal{F}_j}\frac{|P_\L \phi|_t^2}{\|P_\L \phi\|^2}.
\]
As 
\[
       \|P_\L\phi\|^2=\|\phi\|^2-2\Re \langle \phi,\phi-P_\L \phi \rangle
+\|\phi-P_\L\phi \|^2 \geq
       1-\mu^j_\L (t), 
\]
we get
\begin{equation} \label{bound_on_F}
     F_\L^j(t)^2\leq \max _{\phi\in \mathcal{F}_j}
\frac{|\phi|_t^2}{1-\mu_\L^j(t)}=
     \max _{\sum |c_k|^2=1} \frac{\sum_{k=1}^j |c_k|^2
\dnea_k(t)^2}{1-\mu_\L^j(t)}=\frac{\dnea_j(t)^2}{1-\mu_\L^j(t)}.
\end{equation}
Finally, \eqref{bound_on_F} and \eqref{bound_on_mu_L} yield
\begin{align}
F_\L^j(t)^2-\dnea_j(t)^2 &\le \frac{\mu_\L^j(t)}{1-\mu_\L^j(t)}\dnea_j(t)^2
\nonumber\\
&\le 2 \mu_\L^j(t)\dnea_j(t)^2\nonumber\\
&\le 2\frac{3}{\dnea_1(t)^2}\dnea_j(t)^2\sum_{k=1}^j\varepsilon_k^2.
\end{align}
The proof is completed by observing that $F_\L^j(t)+\dnea_j(t)\ge 2\dnea_j(t).$
\end{proof}

As the next corollary shows, a quadratic order of decrease for 
$F^j_\L(t)-\dnea_j(t)$ is prevented for $t\in \spec(A)$ in the context of 
theorems~\ref{linear_order} and \ref{thm:approx_square},  only for $j$ up to 
$\dim \E_{t}(A)$.

\begin{cor}\label{cor_th11}
 Let $t\in\spec_{\disc}(A)$, $\ell=1+\dim \E_{t}(A)$ and $k\in 
\{\ell,\ldots,m\}$. Let
 \[
   \alpha_k(t)= \frac14 \min\left\{|\dnea_l(t)-\dnea_{l-1}(t)|: \dnea_l(t)\not= 
\dnea_{l-1}(t), \, l=\ell,...,k\right\}>0.
\]
There exists $\varepsilon>0$ independent of $k$ ensuring the following. If  
\eqref{cond_approx_1} holds true for $\sqrt{\sum_{j=1}^m 
\varepsilon_j^2}<\varepsilon$, then
\[
     F^k_\mathcal{L}(t)-\dnea_k(t)\leq
3\frac{\dnea_k(t)}{\alpha_k(t)^2} \sum_{j=1}^k \varepsilon_j^2.
\]
\end{cor}

\begin{proof} 
Without loss of generality we assume that $t+\dnea_k(t)\in\spec(A)$. Otherwise 
$t-\dnea_k(t)\in\spec(A)$ and the proof is analogous to the one presented 
below. 

Let $\tilde{t}=t+\alpha_k(t)$. Then   $\tilde{t}\not\in \spec(A)$ and
$t+\dnea_k(t)=\tilde{t}+\dnea_k(\tilde{t})$. Since the map $s\mapsto 
s+F_\L^j(s)$ is non-decreasing as a consequence of Proposition 
\ref{maybe_useless},  Theorem~\ref{thm:approx_square} applied at $\tilde{t}$ 
yields
\begin{align*}
F_\L^k(t)-\dnea_k(t) &= t+F_\L^k(t)-(t+\dnea_k(t)) \le 
\tilde{t}+F_\L^k(\tilde{t})-(\tilde{t}+\dnea_k(\tilde{t}))\\
&= F_\L^k(\tilde{t})-\dnea_k(\tilde{t}) \le 
3\frac{\dnea_k(\tilde{t})}{\dnea_1(\tilde{t})^2}\sum_{j=1}^k\varepsilon_k^2\le  
3
\frac{\dnea_k(t)}{\alpha_k(t)^2}\sum_{j=1}^k\varepsilon_j^2
\end{align*}
as needed.
\end{proof}

\subsection{Convergence of local bounds for eigenvalues}
For the final part of this section, we formulate precise statements on the 
convergence of the method described in Section~\ref{zime_sec}. 
Theorem~\ref{convergence} below improves upon two crucial aspects of a
similar result established in \cite[Lemma~2]{BouStr:2011man}. It allows $j>1$ 
and  it allows $t\in \spec(A)$. These two improvements are essential in order to 
obtain sharp bounds for those
eigenvalues which are either degenerate or form a tight
cluster. 

\begin{remark} \label{on_order_when_in_spectrum}
The constants $\tilde{\varepsilon}_t$ and $C_t^\pm$ below do have a dependence 
on $t$ that may be determined explicitly from Theorem~\ref{thm:approx_square}, 
Corollary~\ref{cor_th11} and the proof of Theorem~\ref{convergence}. Despite of 
the fact that they can deteriorate as $t$ approaches the isolated eigenvalues of 
$A$ and they can have jumps precisely at these points, they may be chosen 
locally independent of $t$ in compacts outside the spectrum. \end{remark}

Set
\begin{align*}
    \nu_j^-(t)&=\sup\{ s<t : \tr \1_{(s,t)}(A)\geq j\} \\
    \nu_j^+(t)&=\inf\{ s>t : \tr \1_{(t,s)}(A)\geq j\}. 
\end{align*}
Note that these are the spectral points of $A$ which are strictly to the left 
and strictly to the right of
$t$ respectively. The inequality
$\nu_j^\pm(t)\not= \nea^\pm_j(t)$ only occurs when $t$ is an eigenvalue. In view 
of \eqref{delta_plus_minus},
 $\delta^{\pm}(t)=|t-\nu_1^\pm(t)|$.

\begin{remark}   \label{rq:SpecDistZM}
By virtue of Corollary~\ref{corollary_4} and Corollary~\ref{prop_useless_in_spectrum},  $\frac{1}{\tau_j^-(t)}\leq \nu_j^-(t)-t$ and $\frac{1}{\tau_j^+(t)}\geq \nu_j^+(t)-t$. Then 
\[
 \hat{t}^-_j=t+\frac{1}{2\tau_j^-(t)}\leq \frac{t+\nu_j^-(t)}{2}
\leq\frac{\nu_j^+(t)+\nu_j^-(t)}{2}
\leq\frac{\nu_j^+(t)+t}{2}\leq t+\frac{1}{2\tau_j^+(t)} =\hat{t}^+_j.
\]
\end{remark}

The following is one of the main results of this paper.

 \begin{theorem} \label{convergence}
 Let $J\subset\R$ be a bounded open segment such that $J\cap \spec(A)\subseteq 
\spec_\disc(A)$. 
 Let $\{\phi_k\}_{k=1}^{\tilde{m}}$ be a family of eigenvectors of $A$ such that 
$\operatorname{Span}\{\phi_k\}_{k=1}^{\tilde{m}}=\E_J(A)$.  For fixed 
$t\in J$,  there exist constants $\tilde{\varepsilon}_t>0$ and $C^{-}_t>0$ independent of 
the trial space $\mathcal{L}$, ensuring the following. If there are 
$\{w_j\}_{j=1}^{\tilde{m}}\subset \mathcal{L}$ such that
\begin{equation} \label{cond_approx_Z-M} 
\left(\sum_{j=1}^{\tilde{m}}  \|w_j-\phi_j\|^2+|w_j-\phi_j|_t^2\right)^{1/2} 
\leq \varepsilon <\tilde{\varepsilon}_t,
\end{equation}
then
\begin{align*}
   0<  \nu^-_j(t)-\left(t+\frac{1}{\tau_j^-(t)}\right) \le
C_t^- \varepsilon^2
\end{align*}
for all $j \leq n^-(t)$ such that $\nu^-_j(t)\in J$. 
\end{theorem}
\begin{proof}  
The hypotheses ensure that the number of indices $j\leq n^-(t)$ such that 
$\nu^-_j(t)\in J$
never exceeds $\tilde{m}$. Therefore this condition in the conclusion of the 
theorem is consistent. 

Let
\[
     m(t)=\max\{m\in \mathbb{N}:[t-\dnea_m(t),t+\dnea_m(t)]\subset J\}.
\]
The hypothesis on $\L$ guarantees that  
\eqref{cond_approx_0}-\eqref{cond_approx_1} hold true for $m=m(t)$ 
and with
$\left(\sum_{j=1}^{m}\varepsilon_j^2\right)^{1/2}
<\varepsilon$.  
By combining  Lemma~\ref{onb_approx} and Theorem~\ref{linear_order} 
and the 
fact that  
we can pick $\{w^t_j\}_{j=1}^{m(t)}\subseteq\{w_k\}_{k=1}^{\tilde{m}}$, 
there exists $\tilde{\varepsilon}_t>0$
small enough,  such that  \eqref{cond_approx_Z-M} yields
\begin{equation}\label{conse-uniform}
     F^j_\L(s)-\dnea_j(s)\leq {\frac{t-\nu_1^-(t)}{2}} \qquad \qquad \forall
j=1,\ldots,\tilde{m}\quad \text{and} \quad s\in J.
\end{equation}

Let $j$ be such that $\nu^-_j(t)\in J$.
Since $ \nu_j^-(t)-(\alpha+t) \le (t+\alpha)- \nu_1^-(t) $ for all 
 $\alpha$ such that $\frac{\nu_j^-(t)+\nu_1^-(t)}{2}-t\le \alpha 
\le 0$, then
\[
    \dnea_j(s)=s-\nu_j^-(t) \qquad \forall  s\in 
\left[\frac{\nu_1^-(t)+\nu_j^-(t)}{2},\frac{t+\nu_j^-(t)}{2}\right].
\]
Let 
\[
     g(\alpha )= F_\L^j(t+\alpha)+\alpha.
\]
Then $g$ is an increasing function of $\alpha$ and $g(0)=F_\L^j(t)>0$. For the 
strict inequality in the latter, recall Assumption~\ref{asu2}.
Moreover, according to \eqref{conse-uniform} 
\begin{align*}
g\left(\frac{\nu_j^-(t)+\nu_1^-(t)}{2}-t\right) 
&=  F_\L^j\left( \frac{\nu_j^-(t)+\nu_1^-(t)}{2}\right) -t+\nu_1^-(t) - 
\frac{\nu_1^-(t)-\nu_j^-(t)}{2} \\
&= F_\L^j\left( \frac{\nu_j^-(t)+\nu_1^-(t)}{2}\right)-t+\nu_1^-(t) - 
\dnea_j\left( \frac{\nu_j^-(t)+\nu_1^-(t)}{2}\right) \\
&\le   \frac{t-\nu_1^-(t)}{2} -(t-\nu_1^-(t)) < 0\,.
\end{align*}
Hence, the Mean Value Theorem ensures the existence of $\tilde{\alpha}\in 
\left(\frac{\nu_1^-(t)+\nu_j^-(t)}{2}-t,0\right)$ such
that $\tilde{\alpha}=F_\L^j(t+\tilde{\alpha})$. According to 
Theorem~\ref{more_useful}, $\tilde{\alpha}$ is unique and 
$\tilde{\alpha}=\frac{1}{2\tau_j^-(t)}$. 

The proof is now completed as follows. {By virtue of Remark 
\ref{rq:SpecDistZM},} 
\[
\hat{t}_j^-(t)=t+\frac{1}{2\tau_j^-(t)}\in 
\left(\frac{\nu_1^-(t)+\nu_j^-(t)}{2}, \frac{t+\nu_j^-(t)}{2}\right)
\quad \text{and} \quad F_\L^j(\hat{t}_j^-(t))=\frac{1}{2\tau_j^-(t)}.
\]
Then, Theorem~\ref{thm:approx_square} or
Corollary~\ref{cor_th11}, as appropriate, ensure the existence of $C_t^->0$ 
yielding
\[
\nu_j^-(t)-\left(t+\frac{1}{\tau_j^-(t)}\right) ={F_\L^j}(\hat{t}
_j^-)-\dnea_j(\hat{t}^-_j)\leq 
C_t^-\sum_{k=1}^j\varepsilon_k^2<C_t^- \varepsilon^2\,,
\]
as needed.
\end{proof}

\subsection{Convergence for eigenfunctions} We conclude this section with a 
statement on convergence of eigenfunctions.

\begin{cor}    \label{conv_eigenfunctions}
Let $J\subset\R$ be a bounded open segment such that $J\cap \spec(A)\subseteq 
\spec_\disc(A)$. 
 Let $\{\phi_k\}_{k=1}^{\tilde{m}}$ be a family of eigenvectors of $A$ such that 
$\operatorname{Span}\{\phi_k\}_{k=1}^{\tilde{m}}=\E_J(A)$.  For fixed $t\in J$, 
there exist constants $\tilde{\varepsilon}_t>0$ and $C^{\pm}_t>0$ independent of 
the trial space $\mathcal{L}$, ensuring the following.  If there are 
$\{w_j\}_{j=1}^{\tilde{m}}\subset \mathcal{L}$ guaranteeing
the validity of \eqref{cond_approx_Z-M}, for all $j\le n^\pm(t)$ such that 
$\nu_j^\pm(t)\in J$ we can find 
$\psi_j^{\varepsilon\pm}\in\E_{\{\nu_j^-(t),\nu_j^+(t)\}}(A)$ satisfying
\[
    |u_j^\pm(t)-\psi_j^{\varepsilon\pm}|_t+ 
\|u_j^\pm(t)-\psi_j^{\varepsilon\pm}\|\leq
C_t^\pm \varepsilon.
\]
\end{cor}
\begin{proof}
Fix $t\in J$. By virtue of Theorem~\ref{more_useful},
$
        u^\pm_j(t)=u^{\hat{t}^\pm_j}_j
$
in the notation for eigenvectors employed in Proposition~\ref{eigenfunctions}.
The claimed conclusion is a consequence of the latter combined with 
Theorem~\ref{thm:approx_square} or
Corollary~\ref{cor_th11}, as appropriate.
\end{proof}


\section{The finite element method for the Maxwell eigenvalue problem}  
\label{section_maxwell}
Let $\Omega\subset \RR^3$  be a polyhedron which is open, bounded, simply 
connected and Lipschitz
 in the sense of \cite[Notation~2.1]{ABDG98}. Let  $\bomega$ be the boundary of 
$\Omega$ and denote by $\nn$ its outer normal vector.  The physical phenomenon 
of electromagnetic oscillations in a resonator filled with a homogeneous medium 
is described by
the isotropic Maxwell eigenvalue problem,
 \begin{equation} \label{maxwell}
\left\{
\begin{aligned}
 & \curl \bE = i\omega \bH & \text {in }\Omega \\
 & \curl \bH = -i\omega \bE  & \text{in }\Omega  \\
 & \bE\times\nn =0& \text{on } \bomega.
 \end{aligned} \right.
\end{equation}
Here the angular frequency $\omega\in \R$ and the field phasor $(\bE,\bH) 
\not=0$ is restricted to the solenoidal subspace,  characterized by the Gauss 
law
\begin{equation}   \label{ansatz_div}
\div (\bE)=0=\div (\bH). 
\end{equation}
The orthogonal complement of this subspace is the gradient space, which has 
infinite dimension and it lies in the kernel of the eigenvalue equation 
\eqref{maxwell}. In turns, this means that \eqref{maxwell}-\eqref{ansatz_div} 
and the unrestricted problem \eqref{maxwell}, have the same non-zero spectrum 
and the same corresponding eigenspace.

Let
\[
   \sobol(\curl;\Omega)=\left\{\bu \in [L^2(\Omega)]^3 : \curl \bu \in
[L^2(\Omega)]^3 \right\}
\]
equipped with the norm
\begin{equation} \label{norm_curl}
    \| \bu \|_{\curl,\Omega}^2=\|\bu\|_{0,\Omega}^2+\|\curl\bu\|_{0,\Omega}^2.
\end{equation}
Let $\mathcal{R}_{\max}$  denote the  operator defined by
the expression ``$\curl$'' acting on the domain
$\dom(\mathcal{R}_{\max})=\sobol(\curl;\Omega)$, the maximal domain.
Let
\[\mathcal{R}_{\min}=\mathcal{R}_{\max}^*=\overline{\mathcal{R}_{\max}
\!\upharpoonright \! [\D(\Omega)]^3}.\] The domain of $\mathcal{R}_{\min}$ is 
\begin{align*}
\dom(\mathcal{R}_{\min})&= \sobol_0(\curl;\Omega) \\
& = \{ \bu \in \sobol(\curl;\Omega) :\langle \curl \bu,\bv\rangle_\Omega =
\langle \bu,\curl \bv \rangle
_\Omega \quad \forall \bv \in \sobol(\curl;\Omega)\}.
\end{align*}
By virtue of Green's identity for the rotational 
\cite[Theorem~I.2.11]{1986Giraultetal}, 
\[
   \sobol_0(\curl;\Omega)=\{ \bu\in \sobol(\curl;\Omega):  \bu\times\nn 
={\mathbf 0} \;\mathrm{on}\;\bomega  \}.
\]
The linear operator associated to \eqref{maxwell} is then,
\[
     \opmax =\begin{pmatrix} 0 & i \mathcal{R}_{\max} \\ -i \mathcal{R}_{\min}
& 0\end{pmatrix}
\]
on the domain
\begin{equation}\label{dom_per}
    \dom(\opmax)=\dom(\mathcal{R}_{\min})\times
\dom(\mathcal{R}_{\max})\subset [L^2(\Omega)]^6.
\end{equation}
Note that $\opmax:\dom(\opmax_1)\longrightarrow [L^2(\Omega)]^6$ is 
self-adjoint,
as $\mathcal{R}_{\max}$ and $\mathcal{R}_{\min}$ are mutually adjoints 
\cite[Lemma~1.2]{1990Birman}.  

The numerical estimation of the eigenfrequencies of 
\eqref{maxwell}-\eqref{ansatz_div} is known to be extremely challenging in 
general.   The operator $\mathcal{M}$  does not have a compact resolvent and it 
is strongly indefinite. The self-adjoint operator associated to 
\eqref{maxwell}-\eqref{ansatz_div} has a 
compact resolvent but it is still strongly indefinite. 
By considering the square of $\mathcal{M}$ on the solenoidal subspace, one 
obtains a positive definite eigenvalue problem (involving the bi-curl) 
which can in principle  be discretized via the Galerkin method. A serious 
drawback of this idea for practical computations is the fact that 
the standard finite element spaces are not solenoidal. Usually, spurious modes 
associated to the infinite-dimensional
kernel appear and give rise to spectral pollution. This has been well documented 
and it is known to be a manifested problem whenever the underlying mesh  is 
unstructured,  \cite{Arnold:2010p3067,Boffi-Act-Num} and references therein.   

Various ingenious methods, e.g. \cite{BG11,Bramble05,BCJ09,BFGP1999, 
Boffi-Act-Num}, 
capable of approximating the eigenvalues of \eqref{maxwell} by means of the 
finite element method have been documented in the past.  Let us apply the 
framework of Section~\ref{zime_sec} for finding eigenvalue bounds for  
$\mathcal{M}$ employing Lagrange finite elements on unstructured meshes.  
Convergence and absence of
spectral pollution are guaranteed, as a consequence of 
Corollary~\ref{corollary_4} and Theorem~\ref{convergence}.

Let $\{\mathcal{T}_h\}_{h>0}$ be a family of shape-regular triangulations of 
$\overline{\Omega}$ \cite{EG04}, where each element
$K\in {\mathcal{T}}_h$ is a simplex  with diameter $h_K$ such that 
$h=\max_{K\in{\mathcal{T}}_h}h_K$.  For $r\ge 1$, let 
\begin{align*}
    \mathbf{V}_h^r &=\{\bv_h\in [C^0(\overline{\Omega})]^3: \bv_h|_K \in
[\mathbb{P}_r(K)]^3 \ 
    \forall K\in \mathcal{T}_h \},  \\
    \mathbf{V}_{h,0}^r &=\{\bv_h\in \mathbf{V}_h^r: \bv_h\x\nn ={\mathbf 0}
\;\textrm{on}\;\bomega \}
\end{align*}
and set
\begin{equation*}\label{fe-space}
 \L_h=\mathbf{V}_{h,0}^r\x \mathbf{V}_h^r \subset 
  \dom(\opmax).
\end{equation*}
Let $ \omega_1\leq \omega_2 \leq \ldots$ be the positive eigenvalues of 
$\mathcal{M}$. The upper bounds $\omega_j^+$ and lower bounds $\omega^-_j$ 
reported below are found by fixing $t\in \mathbb{R}$, solving  (Z$_t^{\L_h}$) 
numerically, and then applying \eqref{encl_gen}. 

The only hypothesis in the analysis carried out above ensuring that the 
$\omega^\pm_j$ are close to $\omega_j$, is for the trial space to capture well 
the eigenfunctions in the graph norm of $\dom(\opmax)$. Therefore, as we have 
substantial freedom to choose these spaces and they constitute the simplest 
alternative, 
we have picked the Lagrange nodal elements. A direct application of Theorem 12 
and classical interpolation estimates e.g. \cite[Theorem~3.1.6]{1978Ciarlet}, 
leads to 
 convergence of  the approximated eigenvalues and eigenspaces. Moreover, if 
the eigenspaces are regular, then the optimal 
convergence rates of order $h^{2r}$ for eigenvalues and
$h^r$ for eigenspaces can be proved.

This regularity assumption 
on the corresponding vector spaces can be formulated in different ways 
in order to suit the chosen algorithm. For the one we have employed here, 
if we wish to obtain a lower/upper bound for the $j$-eigenvalue to the 
left/right of a fixed $t$ (and 
consequently obtain approximate eigenvectors) all the vectors of the sum of 
all eigenspaces up to $j$ have to be regular. If by some 
misfortune, an intermediate eigenspace does not fullfill this requirement, then 
the algorithm will converge slowly. To circumvent this difficulty, the 
computational procedure can be 
modified in many ways. For instance, it can be allowed to split iteratively 
the initial interval, once it is clear that some accuracy can not be achieved
after a fixed number of steps. 

\subsection{Orders of convergence on a cube}
The eigenfunctions of \eqref{maxwell} are regular in the interior of a convex 
domain. In this case, the \zime method for the resonant cavity problem 
 achieves an optimal order of convergence in the 
context of finite elements.

\begin{figure}[t]
\centerline{\includegraphics[height=8cm,
angle=0]{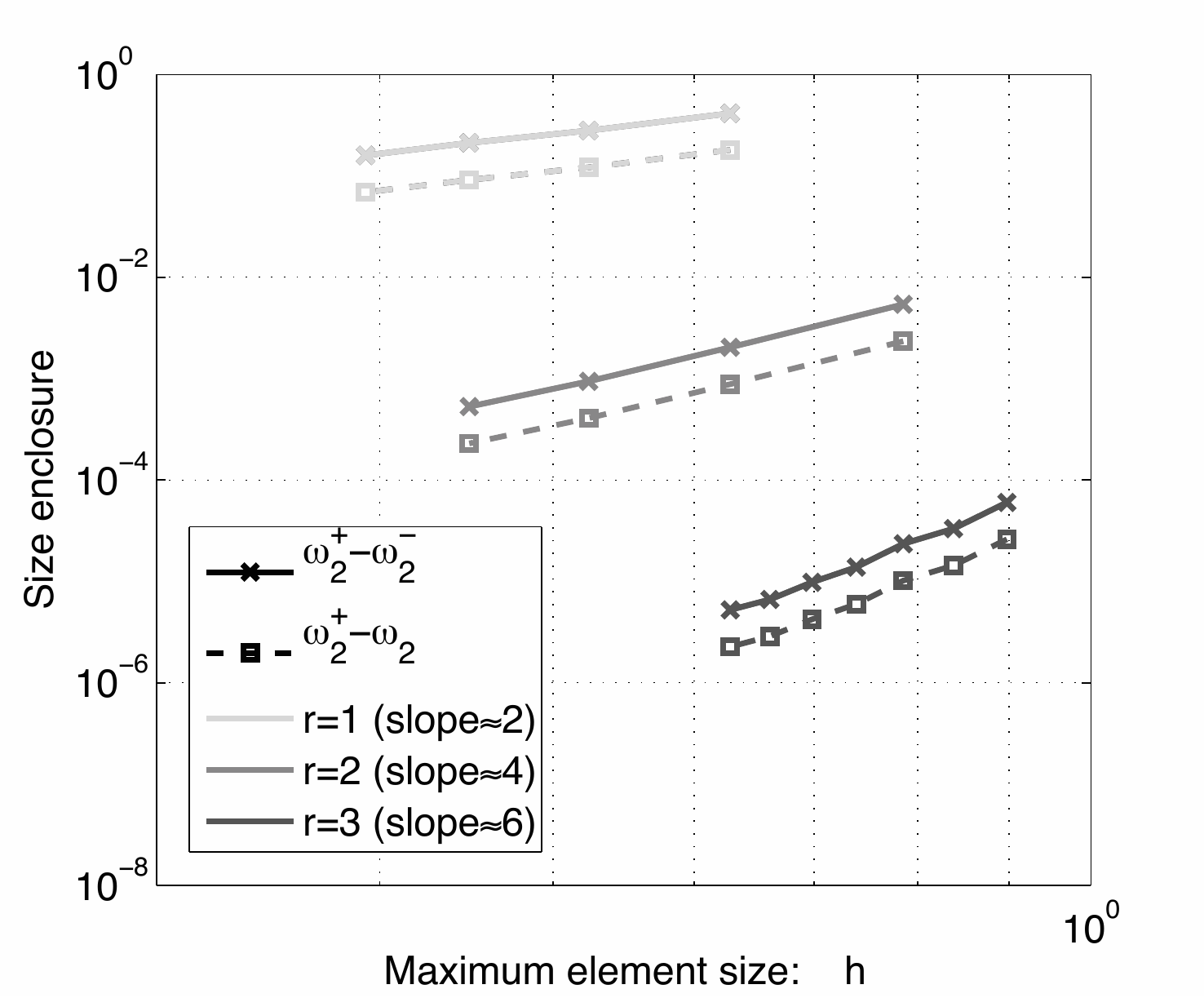}}
\caption{ Log-log graph associated to $\Omega_\cbe$ and $\omega_{2}=\sqrt{3}$.
Vertical axis:  enclosure width. Horizontal axis: maximum element size $h$. Here
we have chosen Lagrange elements of order $r=1,2,3$ on a sequence of
unstructured meshes. Here we have chosen $t=\frac{\sqrt{2}+\sqrt{3}}{2}$ the 
upper bounds and 
$t=\frac{\sqrt{3}+\sqrt{5}}{2}$ for the lower bounds.
\label{cube}
}
\end{figure}

Let $\Omega=\Omega_{\cbe}=(0,\pi)^3\subset \R^3$.  The non-zero eigenvalues
are\[\omega=\pm\sqrt{l^2+m^2+n^2}\] and the  corresponding eigenfunctions are
\[
 \bE(x,y,z)=\begin{pmatrix}   \alpha_1 \cos(lx) \sin(my) \sin(nz) \\ 
 \alpha_2 \sin(lx) \cos(my) \sin(nz) \\ \alpha_3 \sin(lx) \sin(my) \cos(nz)   
\end{pmatrix}
\qquad \forall \underline{\alpha}:=\begin{pmatrix} \alpha_1 \\ 
\alpha_2 \\ \alpha_3
\end{pmatrix}\mbox{ s.t. }
 \underline{\alpha}\cdot \begin{pmatrix} l \\ m \\ n \end{pmatrix} =0.
\]
Here $\{l,m,n\}\subset \mathbb{N}\cup \{0\}$ and not two indices are allowed to
vanish simultaneously.  The vector $\underline{\alpha}$
determines the multiplicity of the eigenvalue for a given triplet $(l,m,n)$. 
That
is, for example,
$\omega_1=\sqrt{2}$ (the first positive eigenvalue) has multiplicity 3
corresponding to indices 
$\{(1,1,0),(0,1,1),(1,0,1)\}$ each one of them contributing to one of the
dimensions of the eigenspace. However, 
$\omega_2=\sqrt{3}$ (the second positive eigenvalue) corresponding to index
$\{(1,1,1)\}$ has multiplicity 2
determined by $\underline{\alpha}$ on a plane.

In Figure~\ref{cube} we have depicted the decrease in enclosure width and exact 
residual,
\[
     \omega^+_2-\omega^-_2 \qquad \text{and} \qquad \omega^+_2-\omega_2,
\]
for the computed bounds of the eigenvalue $\omega_2=\sqrt{3}$ by means of 
Lagrange elements of order
$r=1,2,3$. In this experiment we have chosen a sequence of unstructured 
tetrahedral mesh.
The values for the slopes of the straight lines indicates that the
enclosures obey the estimate of the form
\begin{equation} \label{conjecture_optimal}
        |\omega^\pm-\omega| \le c h^{2r}, 
\end{equation}
which is indeed the optimal convergence rate.

\begin{figure}[t]
\begin{minipage}{6cm}
\includegraphics[height=6cm, angle=0]{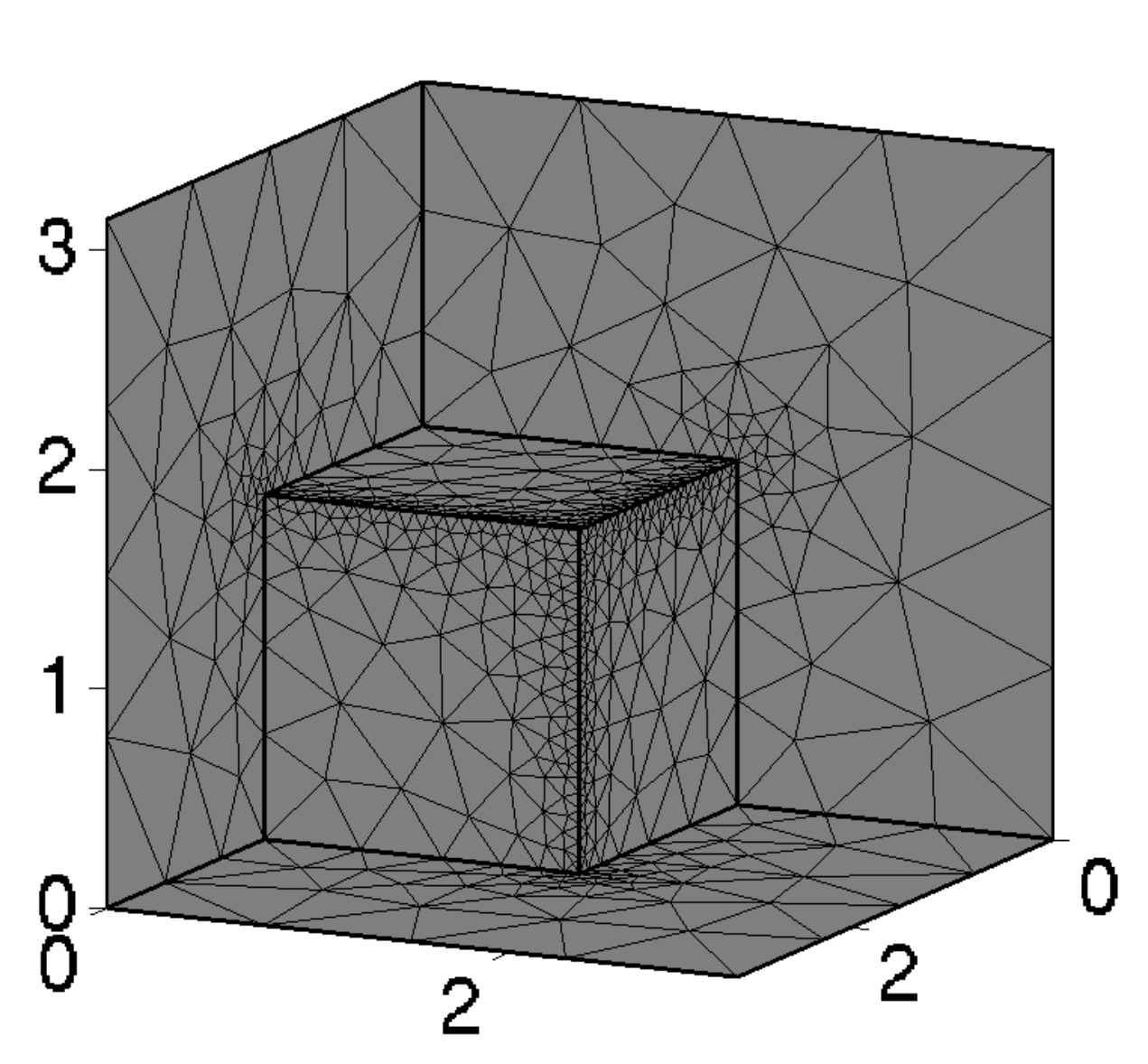}
\end{minipage} \hspace{1cm}
\begin{tabular}{c|c}
$k$  & $(\omega_k)_-^+$ \\  
\hline 
$1$  & $1.1^{46}_{25}$   \\ \hline
$2$ & $1.54^{41}_{16}$ \\
$3$  & $1.54^{41}_{18}$ \\   \hline
$4$  &  $2.0^{82}_{64}$\\
$5$  &  $2.0^{82}_{78}$  \\
$6$ & $2.0^{82}_{78}$\\   \hline
$7$  & $2.2^{35}_{13}$ \\
$8$  & $2.2^{35}_{14}$ \\   \hline
$9$ &  $2.32^{67}_{58}$\\   \hline
$10$  & $2.33^{23}_{09}$ \\
$11$  &  $2.33^{24}_{10}$\\   \hline
$12$  & $2.^{40}_{36}$\\   \hline
$13$  &  $2.^{60}_{59}$\\
$14$  &  $2.^{60}_{59}$\\
$15$  &  $2.60^{56}_{09}$\\
\hline
\end{tabular} \hspace{1cm}
\caption{Spectral enclosures for the spectrum lying on the interval
$(0,2\sqrt{2})$ for the Fichera domain $\Omega_{\lstd}$. Here we have  fixed 
$t=0.2$ to compute the upper bounds and
$t=2.8$ to compute the lower bounds.  We considered mesh refined at the 
re-entrant edges as shown on
the left. The number of DOF=208680. \label{fichera_table}}
\end{figure}

\subsection{Benchmark eigenvalue bounds for the Fichera domain}  
In this next experiment we consider the region
$\Omega=\Omega_{\lstd}=(0,\pi)^3 \setminus [0,\pi/2]^3$. Some of the eigenvalues 
can be obtained by domain decomposition and the
corresponding eigenfunctions are regular. For example, eigenfunctions on the
cube of side $\pi/2$ can be assembled in the obvious fashion, in order to build
eigenfunctions on $\Omega_{\lstd}$. Therefore the set $\{\pm
2\sqrt{l^2+m^2+n^2}\}$ where not two indices vanish simultaneously certainly 
lies
inside $\spec(\opmax)$. The first eigenvalue in this set is $2\sqrt{2}$.

We conjecture that there are exactly $15$ eigenvalues in the interval
$(0,2\sqrt{2})$.  Furthermore, we conjecture that the multiplicity counting of
the spectrum in this interval is
\[
        1,\,2,\,3,\,2,\,1,\,2,\,1,\,3.
\]
 The table on the  right of Figure~\ref{fichera_table} shows a numerical
estimation of these eigenvalues.  We have considered a
mesh refined along the re-entrant edges as shown on
the left side of this figure.  

The slight numerical discrepancy shown in the table for the seemingly multiple
eigenvalues appears to be a consequence of the fact that the meshes employed are 
not entirely
symmetric with respect to permutation of the spacial coordinates.


\pagebreak

\appendix    

\section{A Comsol v4.3 LiveLink code} \label{code}

\begin{verbatim}
%     Comsol V4.3 LiveLink code for computing
%   fundamental frequencies on a resonant cavity
%      with perfect conductivity conditions
% the test geometry below is the Fichera domain.
%
%      Gabriel Barrenechea, Lyonell Boulton
%              and Nabile Boussaid
%                                       November 2012

% INITIALIZATION OF THE MODEL FROM SCRATCHES

model = ModelUtil.create('Model');
geom1=model.geom.create('geom1', 3);
mesh1=model.mesh.create('mesh1', 'geom1');
w=model.physics.create('w', 'WeakFormPDE', 'geom1', 
              {'E1','E2', 'E3', 'H1', 'H2', 'H3'});

% CREATING THE GEOMETRY - IN THIS CASE THE FICHERA DOMAIN

hex1=geom1.feature.create('hex1', 'Hexahedron');
hex1.set('p',{'0' '0' '0' '0' 'pi' 'pi' 'pi' 'pi';
              '0' '0' 'pi' 'pi' '0' '0' 'pi' 'pi';
              '0' 'pi' 'pi' '0' '0' 'pi' 'pi' '0'});
hex2=geom1.feature.create('hex2', 'Hexahedron');
hex2.set('p',{'0' '0' '0' '0' 'pi/2' 'pi/2' 'pi/2' 'pi/2';
              '0' '0' 'pi/2' 'pi/2' '0' '0' 'pi/2' 'pi/2';
              '0' 'pi/2' 'pi/2' '0' '0' 'pi/2' 'pi/2' '0'});
dif1 = geom1.feature.create('dif1', 'Difference');
dif1.selection('input').set({'hex1'});
dif1.selection('input2').set({'hex2'});
geom1.run;

%CREATING THE GEOMETRY
model.mesh('mesh1').automatic(false);
model.mesh('mesh1').feature('size').set('custom', 'on');
model.mesh('mesh1').feature('size').set('hmax', '.8');
mesh1.run;

% PARAMETER t WHERE TO LOOK FOR EIGENVALUES
parat=2.2;

% WHETHER TO LOOK FOR THE EIGENVALUES TO THE LEFT (-) OR 
% RIGHT (+) AND WHERE ABOUT
shi=-.3;
model.param.set('tt', num2str(parat));
searchtau=shi;

% FINITE ELEMENTS TO USE AND ORDER
w.prop('ShapeProperty').set('shapeFunctionType', 'shlag');
w.prop('ShapeProperty').set('order', 3);

% PHYSICS
w.feature('wfeq1').set('weak',1 ,'(H3y-H2z)*(H3y_test-H2z_test)-
i*2*tt*(H3y-H2z)*E1_test+tt^2*E1*E1_test+(i*(H3y-H2z)-tt*E1)*E1t_test');
w.feature('wfeq1').set('weak',2 ,'(H1z-H3x)*(H1z_test-H3x_test)-
i*2*tt*(H1z-H3x)*E2_test+tt^2*E2*E2_test+(i*(H1z-H3x)-tt*E2)*E2t_test');
w.feature('wfeq1').set('weak',3 ,'(H2x-H1y)*(H2x_test-H1y_test)-
i*2*tt*(H2x-H1y)*E3_test+tt^2*E3*E3_test+(i*(H2x-H1y)-tt*E3)*E3t_test');
w.feature('wfeq1').set('weak',4 ,'(E3y-E2z)*(E3y_test-E2z_test)+
i*2*tt*(E3y-E2z)*H1_test+tt^2*H1*H1_test+((-i)*(E3y-E2z)-tt*H1)*H1t_test');
w.feature('wfeq1').set('weak',5 ,'(E1z-E3x)*(E1z_test-E3x_test)+
i*2*tt*(E1z-E3x)*H2_test+tt^2*H2*H2_test+((-i)*(E1z-E3x)-tt*H2)*H2t_test');
w.feature('wfeq1').set('weak',6 ,'(E2x-E1y)*(E2x_test-E1y_test)+
i*2*tt*(E2x-E1y)*H3_test+tt^2*H3*H3_test+((-i)*(E2x-E1y)-tt*H3)*H3t_test');

% BOUNDARY CONDITIONS
cons1=model.physics('w').feature.create('cons1', 'Constraint');
cons1.set('R', 2, 'E2');
cons1.set('R', 3, 'E3');
cons1.selection.set([1 8 9]);
cons2=model.physics('w').feature.create('cons2', 'Constraint');
cons2.set('R', 1, 'E1');
cons2.set('R', 3, 'E3');
cons2.selection.set([2 5 7]);
cons3=model.physics('w').feature.create('cons3', 'Constraint');
cons3.set('R', 1, 'E1');
cons3.set('R', 2, 'E2');
cons3.selection.set([3 4 6]);

% HOW MANY EIGENVALUES TO LOOK FOR AROUND t
neval=3;

% SOLVING THE MODEL
std1=model.study.create('std1');
model.study('std1').feature.create('eigv', 'Eigenvalue');
model.study('std1').feature('eigv').set('shift', num2str(searchtau));
model.study('std1').feature('eigv').set('neigs', neval);
std1.run;

% STORING SOLUTION FOR POST PROCESSING
[SZ,NDOFS,DATA,NAME,TYPE]= mphgetp(model,'solname','sol1');



% DISPLAYING SOLUTION
for inde=1:neval,
tauinv=(real(DATA(inde)));
bd=parat+tauinv;
if tauinv<0, disp(['lower= ',num2str(bd,10)]);
else disp(['upper= ',num2str(bd,10)]);
end
disp(['DOF= ',num2str(NDOFS)])
end
\end{verbatim}


\section*{Acknowledgements} 
We kindly thank Michael Levitin and Stefan Neuwirth for their suggestions during 
the preparation of this manuscript. We kindly thank Universit\'e de 
Franche-Comt{\'e}, University College London and the Isaac Newton Institute for 
Mathematical Sciences, for their hospitality. Funding was provided by  MOPNET, 
the British-French project PHC Alliance  (22817YA), the British Engineering and 
Physical Sciences Research Council  (EP/I00761X/1) and the French Ministry of 
Research (ANR-10-BLAN-0101).

\bibliographystyle{siam}
\def\cprime{$'$}

\end{document}